\documentclass[a4paper,10pt,reqno]{amsart}
\usepackage[amsart]{optional}

\usepackage{amssymb, amsmath, amsthm, amsfonts, comment}
\usepackage{bbm}
\usepackage[colorlinks]{hyperref}
\usepackage{tikz-cd}
\usetikzlibrary{matrix,arrows}
\usepackage{float}
\usepackage{setspace} 

\newcommand\itemindentcorrection{0pt}

\opt{amsart}{
\setcounter{tocdepth}{2}
\let\oldtocsection=\tocsection
\let\oldtocsubsection=\tocsubsection
\let\oldtocsubsubsection=\tocsubsubsection
\renewcommand{\tocsection}[3]{\hspace{0em}\oldtocsection{#1}{#2}{#3}}
\renewcommand{\tocsubsection}[3]{ \hspace{1em} \oldtocsubsection{#1}{\small{#2}}{\small{#3}} }
\renewcommand{\tocsubsubsection}[3]{\hspace{2em}\oldtocsubsubsection{#1}{\small{#2}}{\small{#3}}}
}

\newcommand{\quotes}[1]{\textquoteleft{#1}'}
\newcommand{\set}[1]{\left\{{#1}\right\}}

\newcommand{\setconds}[2]{\ensuremath{\left\{\; #1\ :\ #2 \;\right\}}}
\newcommand{\setcondsbig}[2]{\ensuremath{\big\{\; #1\ :\ #2 \;\big\}}}

\newcommand{\quotstack}[2]{\ensuremath{\left[\,#1\,\middle/\,#2\,\right]}}
\newcommand{\quotstackinline}[2]{\ensuremath{[\,#1\,/\,#2\,]}}
\newcommand{\quotstackbig}[2]{\ensuremath{\big[\,#1\,\big/\,#2\,\big]}}
\newcommand{\quotstackBig}[2]{\ensuremath{\Big[\,#1\,\Big/\,#2\,\Big]}}
\newlength\myWindowLength
\newcommand{\WindowGens}[3]{\; #1\ :\ l\in \left[0,#2\right),\; m\in \left[0,#3\right)\,}
\newcommand{\WindowBox}[3]{\settowidth{\myWindowLength}{$\WindowGens{#1}{#2}{#3}$}
\begin{minipage}[c][15pt][c]{\myWindowLength}{$\WindowGens{#1}{#2}{#3}$}\end{minipage}}
\newcommand{\WindowSet}[3]{\left\{ \WindowBox{#1}{#2}{#3} \right\}}
\newcommand{\Window}[3]{\left\langle \WindowBox{#1}{#2}{#3} \right\rangle}

\newcommand\xyhook{\ar@{^{(}->}}
\newcommand\Z{\mathbb Z}
\newcommand\C{\mathbb C}
\newcommand\A{\mathbb A}

\newcommand\id{\operatorname{id}}

\newcommand\into{\hookrightarrow}
\newcommand\To{\longrightarrow}

\DeclareMathOperator \Hom {Hom}
\renewcommand\hom{\mathcal{H}om }
\DeclareMathOperator\Ext{Ext}
\newcommand\RDerived{\mathrm{R}}
\DeclareMathOperator \RHom {RHom}
\renewcommand\Im{\operatorname{Im}}
\renewcommand\P{\mathbb P}
\newcommand{\Gr}{\mathbb{G}\mathrm{r}}
\newcommand{\Pf}{\mathbb{P}\mathrm{f}}

\newcommand\Perf{\operatorname{Perf}}
\newcommand\rk{\operatorname{rank}}

\newcommand{\cE   }{\mathcal{E}}
\newcommand{\cF}{\mathcal{F}}
\newcommand{\cN}{\mathcal{N}}
\newcommand{\cO}{\mathcal{O}}

\newcommand{\cG}{\mathcal{G}}

\newcommand\mfX{\mathfrak X}

\newcommand{\Tot}{\operatorname{Tot}}
\newcommand{\Sym}{\operatorname{Sym}}
\newcommand{\Wedge}{\mbox{\large{$\wedge$}}}

\newcommand{\Crit}{\operatorname{Crit}}
\newcommand{\otherl}{l'}
\newcommand{\otherm}{m'}
\newcommand{\SymOtherl}{\Sym^{\otherl}\!}
\newcommand{\stackProjection}[1]{q_{#1}}
\DeclareMathOperator \SL {SL}
\DeclareMathOperator \GL {GL}

\newcommand \Caldararu {C\u{a}ld\u{a}\-raru}

\newlength\tempWidth
\newcommand\InSpaceOf[2]{
   \settowidth{\tempWidth}{$#1$}\phantom{#1}\hspace{-\tempWidth}{#2}}

\newcommand\Br{\mathcal{BB}r}
\newcommand\pr{\phi}
\newcommand\D{\mathrm{D^b}}
\newcommand\QCoh{\mathrm{QCoh}}
\newcommand\QCohdg{\mathrm{QCoh_{dg}}}
\newcommand\QCohdgnv{\mathrm{QCoh_{dg}^{nv}}}
\newcommand\maxiso{M}
\newcommand\maxisoGenl{M}
\newcommand\opensetinP{U}
\newcommand\opensetinY{Y'}
\newcommand\Schur{\mathbb{S}}
\newcommand\stackFibre{\mathfrak{F}}
\newcommand\twoform[1]{\omega_{#1}}

\numberwithin{equation}{section}

\theoremstyle{plain}
\newtheorem{thm}{Theorem}[section]
\newtheorem{prop}[thm]{Proposition}
\newtheorem{lem}[thm]{Lemma}
\newtheorem{cor}[thm]{Corollary}

\theoremstyle{remark}
\newtheorem{rem}[thm]{Remark}
\newtheorem*{acks}{Acknowledgements}

\theoremstyle{definition}
\newtheorem{defn}[thm]{Definition}
\newtheorem{assm}[thm]{Assumption}
\newtheorem{term}[thm]{Terminology}
\newtheorem{eg}[thm]{Example}

\newcommand{\pgap}{\vspace \baselineskip}

\newcommand{\marginparstretch}{0.6}
\let\oldmarginpar\marginpar
\renewcommand\marginpar[1]{\-\oldmarginpar[\framebox{\setstretch{\marginparstretch}\begin{minipage}{\marginparwidth}{\raggedleft\scriptsize #1}\end{minipage}}]{\framebox{\setstretch{\marginparstretch}\begin{minipage}{\marginparwidth}{\raggedright\scriptsize #1}\end{minipage}}}}

\begin{document}

\title{The Pfaffian-Grassmannian equivalence revisited}

\opt{amsart}{
\author{Nicolas Addington}
\address{Nicolas Addington, Department of Mathematics, Box 90320, Duke University, Durham, NC, 27708-0320, U.S.A.}
\email{adding@math.duke.edu}
\author{Will Donovan}
\address{Will Donovan, Kavli Institute for the Physics and Mathematics of the Universe (WPI), Todai Institutes for Advanced Study, The University of Tokyo, 5-1-5 Kashiwanoha, Kashiwa, Chiba, 277-8583, Japan}
\email{will.donovan@ipmu.jp}
\author{Ed Segal}
\address{Ed Segal, Department of Mathematics, Imperial College London, London, SW7 2AZ, U.K.}
\email{edward.segal04@imperial.ac.uk}
}
\opt{alggeom}{
\author{Nicolas Addington}
\email{adding@math.duke.edu}
\address{Department of Mathematics, Box 90320, Duke University, Durham, NC, 27708-0320, U.S.A.}
\author{Will Donovan}
\email{will.donovan@ipmu.jp}
\address{Kavli Institute for the Physics and Mathematics of the Universe (WPI), Todai Institutes for Advanced Study, The University of Tokyo, 5-1-5 Kashiwanoha, Kashiwa, Chiba, 277-8583, Japan}
\author{Ed Segal}
\email{edward.segal04@imperial.ac.uk}
\address{Department of Mathematics, Imperial College London, London, SW7 2AZ, U.K.}
}

\opt{alggeom}{
\thanks{We thank Matt Ballard for helpful discussions, Ronen Plesser for helpful comments, and the anonymous referee for a thorough reading and many useful suggestions.  N.A. and W.D. are grateful for the hospitality of the Hausdorff Research Institute for Mathematics, Bonn.  W.D.~is grateful for the support of Iain Gordon and EPSRC grant no.~EP/G007632/1, and of the World Premier International Research Center Initiative (WPI Initiative), MEXT, Japan. W.D.~also thanks Paul Aspinwall for supporting a visit to Duke University under NSF grant no.~DMS-0905923.  E.S.~is supported by an Imperial College Junior Research Fellowship.}
}

\begin{abstract}
We give a new proof of the `Pfaffian-Grassmannian' derived equivalence between certain pairs of non-birational Calabi--Yau threefolds. Our proof follows the physical constructions of Hori and Tong, and we factor the equivalence into three steps by passing through some intermediate categories of (global) matrix factorizations. The first step is global Kn\"orrer periodicity, the second comes from a birational map between Landau--Ginzburg B-models, and for the third we develop some new techniques.
\end{abstract}
\subjclass[2010]{Primary 14F05, 14J32, 18E30, 81T30; Secondary 14M15\opt{alggeom}{.}}

\maketitle
\opt{amsart}{\tableofcontents}


\section{Introduction}

The \quotes{Pfaffian-Grassmannian equivalence} refers to a relationship between two particular Calabi--Yau three-folds: $Y_1$, which is a linear section of the Grassmannian $\Gr(2,7)$, and $Y_2$, which is the dual linear section of the Pfaffian locus in $\P(\wedge^2 \C^7)$. The relationship was first conjectured by R\o dland \cite{Rodland}, who by studying their Picard--Fuchs equations observed that $Y_1$ and $Y_2$ appeared to have the same mirror. This means that the usual Conformal Field Theories with these target spaces should occur as different limit points in the K\"ahler moduli space of a single field theory. By itself this is a fairly common phenomenon; the special feature of this example is that $Y_1$ and $Y_2$ are (provably) not birational to one another. This was the first example with this property, and such examples remain extremely rare.

  If we pass to the B-twist of this theory, this picture implies that the B-models defined on $Y_1$ and $Y_2$ are isomorphic, and in particular that their categories of B-branes are equivalent. The category of B-branes on a variety is the derived category of coherent sheaves, so this suggests that we should have a  derived equivalence
\begin{equation}\label{eqn.equivalence1}\D(Y_1) \cong \D(Y_2). \end{equation}
This is a precise mathematical prediction, and it was proven by Borisov and \Caldararu\ \cite{BorCal}, and independently by Kuznetsov \cite{KuznetsovHPDlines} using his broader program of Homological Projective Duality.  

Around the same time as these proofs of \eqref{eqn.equivalence1} appeared, Hori and Tong \cite{HT} wrote an important physics paper that gave an argument for R\o dland's full conjecture,  by constructing the necessary field theory containing $Y_1$ and $Y_2$ in its K\"ahler moduli space. The theory is a Gauged Linear Sigma Model (GLSM), which is a standard idea, but the gauge group is non-abelian, and furthermore the argument that $Y_2$ occurs as a limit relies on some very original analysis of non-perturbative effects.

In this paper we give a new mathematical proof of the derived equivalence \eqref{eqn.equivalence1}, inspired by the ideas of Hori and Tong. In particular we find that this derived equivalence is at heart a birational phenomenon, but the birationality is between two Landau--Ginzburg models
$$\begin{tikzcd} (X_1, W) \arrow[dashed]{r} & \arrow[dashed]{l} (X_2, W). \end{tikzcd}$$
Here $X_1$ and $X_2$ are larger spaces containing $Y_1$ and $Y_2$, and $W$ is a holomorphic function defined on both of them. The space $X_1$ is a variety and $Y_1$ is the critical locus of $W$ in $X_1$, so this we can analyze by standard techniques. However, on the other side we encounter two rather novel phenomena:
\begin{itemize}
\item The space $X_2$ is not a variety; it's an Artin stack.  It seems that the category of B-branes on an Artin stack is \emph{not} the same as the derived category, indeed the correct definition of this category is not known in general.\footnote{See Section~\ref{sect:without} for more discussion of this point.}
\item The subspace $Y_2 \subset X_2$ is not the critical locus of $W$. 
\end{itemize}
We develop new mathematical ideas to handle these phenomena, which very roughly parallel the new physics in \cite{HT}. 

The importance of abelian GLSMs is now fairly widely understood in the mathematics literature, since they are closely connected to toric varieties and complete intersections therein. However, we are only just beginning to understand the world of non-abelian GLSMs. We hope that the perspective and techniques of this paper will encourage others to explore it further.
 
For the remainder of this introduction we explain the constructions that we're going to use, and give an outline of the ideas involved in the proof.

\subsection{Calabi--Yau three-folds}\label{sect:geometric}

Let $V$ be a 7-dimensional complex vector space, and fix a surjective linear map
\[ A\colon \wedge^2 V \to V. \]
From these data we will build two different Calabi--Yau 3-folds:

\begin{list}{$Y_{\arabic{enumi}}$: }{\usecounter{enumi} \setlength\itemsep{5pt} }
\item We consider the Grassmannian
$$\Gr(2, V) \subset \P(\wedge^2 V)$$
in its Pl\"ucker embedding.  Intersecting it with the 7 hyperplanes given by the kernel of $A$, we obtain the first Calabi--Yau 3-fold $Y_1$.

\item We consider the projective space $\P(\wedge^2 V^\vee)$ of 2-forms on $V$. Thinking of a 2-form as an antisymmetric matrix we see that its rank must always be even, so generically the rank is $6$. The Pfaffian locus
$$\Pf(V) \subset \P(\wedge^2 V^\vee)$$
is where the rank drops to 4 or less.  Intersecting this with the linear $\P^6$ given by the image of $A^\vee$, we obtain the second Calabi--Yau 3-fold $Y_2$.
\end{list}

\begin{assm} \label{assm:generic}
We choose $A$ generically enough that the codimension-7 space $\P(\ker A) \subset \P(\wedge^2 V)$ is transverse to $\Gr(2,V)$, so $Y_1$ is smooth.  The smoothness of $Y_2$ is slightly more delicate, since $\Pf(V)$ has singularities along the locus where the rank drops to 2, i.e.\ along $\Gr(2,V^\vee) \subset \Pf(V)$.  But in fact if $\P(\ker A)$ is transverse to $\Gr(2,V)$ then $\P^6 = \P(\Im A^\vee)$ avoids this singular locus and is transverse to the smooth locus of $\Pf(V)$, so $Y_2$ is smooth.  This follows from the fact that $\Pf(V)$ is the (classical) projective dual of $\Gr(2,V)$; details are given in \cite[\S\S 1--2 and especially Cor.~2.3]{BorCal}.\footnote{We are referring to Borisov--\Caldararu\ only for this geometric fact, which does not depend on their proof that $\D(Y_1) \cong \D(Y_2)$.  Our proof of the latter is independent of theirs.}
\end{assm}

\subsection{Relation with the Hori--Tong GLSM}\label{sect:geometricGLSM}

Now we can explain our interpretation of Hori and Tong's construction. Let $S$ be a 2-dimensional complex vector space, and consider the linear Artin stack
$$\mfX = \quotstack{\Hom(S,V)\oplus \Hom( V, \wedge^2 S)}{\GL(S)}.$$
Notice that $\GL(S)$ acts trivially on the determinant of the vector space underlying $\mfX$, so $\mfX$ is a Calabi--Yau stack. 

For Hori and Tong, these data specify a GLSM, which is a certain kind of 2-dimensional supersymmetric gauge theory. It's conformal because of the Calabi--Yau condition. The Lagrangian for this field theory contains a certain parameter $\tau$ (the complexified FI parameter) which is essentially the K\"ahler modulus. The two limits $|\tau|\gg 1$ and $|\tau|\ll 1$ roughly correspond to the two possible GIT quotients of $\mfX$.

\begin{enumerate}
\item In the first limit $|\tau|\gg 1$, we choose a stability condition consisting of a positive character of $\GL(S)$. The unstable locus is where $x$ does not have full rank, and the GIT quotient is the variety
$$X_1 = \Gr(2, V)\times_{\GL(S)} \Hom( V, \wedge^2 S).  $$
This is the total space of the vector bundle $\cO(-1)^{\oplus 7}$ over $\Gr(2, V)$.\footnote{Here and throughout we use the convention that $\cO(-1):=\det(S)=\wedge^2 S$.}  In this limit, the GLSM is  expected to reduce to a sigma model with target $X_1$.

\item Now we look at the other stability condition $|\tau|\ll 1$, where we choose a negative character of $\GL(S)$. At this point we have to be careful about our definition of the GIT quotient. Conventionally, one deletes the unstable locus, then takes the scheme-theoretic quotient of the remaining semi-stable locus. For our purposes this is too destructive, and we will instead take the \emph{stack-theoretic} quotient of the semi-stable locus.\footnote{In fact this is now quite a standard thing to do, particularly if the resulting quotient stack is only an orbifold.} For this stability condition the only unstable points are the locus $p=0$, so we consider the complement
$$X_2 := \set{p\neq 0} \subset \mfX.$$

This space $X_2$ is an Artin stack; we can think of it as a bundle over
$$\P\Hom(V,\wedge^2 S) \cong \P^6$$
 whose fibres are the stacks
$$\quotstackbig{\Hom(S,V)}{\SL(S)}.$$
The classical GIT quotient is the scheme underlying $X_2$: this is singular, and we'll make no use of it. It appears that the Artin stack $X_2$ is the correct space to consider in the $|\tau| \ll 1$ limit. In physics terminology, the gauge group has been broken only to a continuous subgroup. Notice that since the stack $\mfX$ is Calabi--Yau, so too are the open substacks $X_1$ and $X_2$.
\end{enumerate}

The GLSM has another ingredient, known as the `superpotential'. This is the (invariant) function $W$ on $\mfX$ defined by
\begin{equation}\label{eq:W} W(x,p) =  p\circ A\circ \wedge^2 x. \end{equation}
Here $x\in \Hom(S,V) $ and $p\in \Hom(V, \wedge^2 S)$, and $A$ is our fixed linear map from above. We can restrict $W$ to either $X_1$ or $X_2$: the three pairs $(\mfX, W)$, $(X_1, W)$ and $(X_2, W)$ then all define \emph{Landau--Ginzburg B-models} (see \S \ref{sect:categories}).  

The most important thing about a Landau--Ginzburg model is the critical locus of the superpotential $W$. We now indicate how an analysis of this locus for the Landau--Ginzburg B-models $(X_1, W)$ and $(X_2, W)$ will allow us to recover the Calabi--Yau three-folds $Y_1$ and $Y_2$ from the previous Section~\ref{sect:geometric}.

\begin{enumerate}
\item In the case of the pair $(X_1, W)$, we claim that the critical locus of $W$ is exactly our Grassmannian Calabi--Yau $Y_1$. To see this, pick a basis for $V$, so $A$ defines 7 sections $a_1,\ldots,a_7$ of $\cO(1)$ on $\Gr(2,V)$, which we can pull up to $X_1$.  On $X_1$ we also have 7 tautological sections 
$p_1,\ldots,p_7$ of the pullback of $\cO(-1)$, and the superpotential is 
$$W = \sum_{i=1}^7  a_i p_i.$$
Because $A$ is generic, the critical locus of this function is the set \opt{amsart}{\[\set{a_i=p_i=0, \,\forall i},\]}\opt{alggeom}{$\set{a_i=p_i=0, \,\forall i},$} which by definition is $Y_1\subset \Gr(2,V)$.

\item Now consider the pair $(X_2, W)$. If we fix a point $[p] \in \P\Hom(V, \wedge^2 S)$, then $W$ restricts to give a quadratic form $W_p$ on the fibre $\Hom(S,V)$.  The rank of this quadratic form is twice that of the (antisymmetric) form $p\circ A$.  So the Pfaffian Calabi--Yau $Y_2$ is the locus of points $p$ where the quadratic superpotential $W_p$ on the fibre drops in rank. As we shall see in Section~\ref{sect:pfaffian}, this is contained in (but not equal to) the critical locus of $W$.
\end{enumerate}

\subsection{Outline of proof}\label{sect:outline}
 Associated to any Landau--Ginzburg B-model $(Y,W)$ there is a category, which we denote $\D(Y,W)$, whose objects are \quotes{twisted complexes} or \quotes{global matrix factorizations}. In the special case when $W\equiv 0$, this category is the usual derived category of coherent sheaves $\D(Y)$. We will prove the derived equivalence \eqref{eqn.equivalence1} as a composition of three equivalences, as follows:
\begin{equation}\label{threeequivs}
\begin{tikzcd}[column sep=0em,row sep=3.5em]
\D(X_1, W) \arrow{rr}{\Psi_2}[below]{\sim} & & \Br(X_2,W) & \subset & \D(X_2, W)  \\
\D(Y_1)\ar{u}{\Psi_1}[below,rotate=90]{\sim} & \hspace{3em} & \D(Y_2) \ar{u}[swap]{\Psi_3}[above,rotate=90]{\sim}
\end{tikzcd}\end{equation}
Let's say a few words about each step.

\begin{list}{$\Psi_{\arabic{enumi}}$: }{\usecounter{enumi} \setlength\itemsep{5pt} }

\item This step is well-known to experts; it is a generalization of Kn\"orrer periodicity which has been proved several times over in recent years. We explain this step in Section~\ref{sect:knorrer}.

\item Let's forget about $W$ momentarily, and also forget that $X_2$ is an Artin stack. Since they are related by variation of GIT, $X_1$ and $X_2$ are birational Calabi--Yau spaces.  Kawamata and Bondal--Orlov have conjectured that any two birational Calabi--Yau's are derived equivalent, and this is known to be true in many cases.  Putting $W$ back in, a more general conjecture is that birational Calabi--Yau Landau--Ginzburg models have equivalent categories of global matrix factorizations.\footnote{In fact this should follow fairly easily from the $W=0$ case.} 

 However, our $X_2$ is actually an Artin stack. This complicates things, and in fact $\D(X_2)$ is much bigger than $\D(X_1)$. However, as we shall see, we can construct a fully faithful embedding from $\D(X_1)$ into $\D(X_2)$. We denote its image by $\Br(X_2)$, and we postulate that this is the correct category of B-branes for the stack $X_2$.  

When we put $W$ back in we have a corresponding equivalence from $\D(X_1, W)$ to a certain subcategory $\Br(X_2, W)\subset \D(X_2, W)$. We will explain this step in Section~\ref{sect:windows}.
 
\item For Hori and Tong, this is the stage that requires the most novel arguments, and the same is true for us. We use a variation on the Kn\"orrer periodicity argument (as in step~1) to construct an embedding of $\D(Y_2)$ into $\D(X_2,W)$, and show that the  image is the subcategory $\Br(X_2,W)$.  We explain this step in Section~\ref{sect:pfaffian}.
\end{list}

\begin{rem} It would be nice to compare our derived equivalence to the ones found by Borisov--\Caldararu\ and Kuznetsov; unfortunately we do not know how to do this.
\end{rem}

\begin{rem} It may be helpful to compare what we do here to the proof of the `Calabi--Yau/Landau--Ginzburg correspondence' for B-branes presented in \cite{Segal} and \cite{Shipman}. The goal of that project was similarly to re-prove a known equivalence (due to Orlov \cite{OrlovCohSing}) using methods that were more faithful to the original physical arguments. 

Orlov's result is the equivalence
$$\D(Y) \cong \D\big(\quotstack{\C^n}{\Z_n}, f\,\big)$$
where $f$ is a degree $n$ polynomial in $n$ variables, and $Y\subset \P^{n-1}$ is the corresponding Calabi--Yau hypersurface. In the new proof the equivalence is factored into two steps, by considering an abelian gauged linear sigma model
$$  \quotstack{\C^{n+1}}{\C^*} $$
with the superpotential $W=fp$, where $\C^*$ acts with weights $(1,1,\ldots,1,-n)$ and $p$ is the last coordinate. There are two GIT quotients:  the first one is the total space of the canonical bundle $K_{\P^{n-1}}$, and the first step is to prove an equivalence
$$\D(Y) \cong \D\big(K_{\P^{n-1}}, W\big).$$
This follows from a `global Kn\"orrer periodicity' theorem, and we will use exactly the same theorem to deduce our equivalence $\Psi_1$.

The second GIT quotient is the orbifold $\quotstackinline{\C^n}{\Z_n}$, and the second step is to prove an equivalence
$$D^b(K_{\P^{n-1}}, W) \cong D^b\big(\quotstack{\C^n}{\Z_n}, f\,\big).$$
We will extend the methods of this proof to prove our equivalence $\Psi_2$.

Note that there is no analogue of our third step $\Psi_3$ in this construction.
\end{rem}

\begin{rem}
Another previous body of work which is relevant is the study of the derived categories of intersections of quadrics, particularly as retold in \cite{ASS}. There one considers an abelian gauged linear sigma model
$$ \quotstack{\C^{3n}}{\C^*} $$
where the $\C^*$ acts with weight $1$ on the first $2n$ coordinates $x_1, \ldots, x_{2n}$, and with weight $-2$ on the last $n$ coordinates $p_1, \ldots, p_n$. We equip this with a superpotential
$$ W = \sum_{i=1}^n f_i p_i $$
where each $f_i$ is quadratic in the $x$ variables.  The first GIT quotient $X_1$ is the total space of $\cO(-2)^{\oplus n}$ over $\P^{2n-1}$, and global Kn\"orrer periodicity gives an equivalence
$$\D(X_1, W) \cong \D(Y_1)$$
where $Y_1\subset \P^{2n-1}$ is the Calabi--Yau formed by intersecting all the quadrics. The second GIT quotient $X_2$ is the total space of the (orbi-)vector bundle $\cO(-1)^{\oplus 2n}$ over the weighted projective space $\P^{n-1}_{2:2:\ldots:2}$, and one obtains an equivalence
$$\D(X_1, W)\cong \D(X_2, W)$$
by the same methods as before. So we've passed through two steps, which are essentially the same as those in the previous remark.

For the third step, we view $(X_2, W)$ as a family of LG B-models over $\P^{n-1}$, each of which is of the form $\left(\quotstackinline{\C^{2n}}{\Z_2}, W_p\right)$ for some quadratic form $W_p$.\footnote{This point of view is an analogue of the physicists' Born--Oppenheimer approximation.} Where $W_p$ is non-degenerate, Kn\"orrer periodicity tells us that the category of matrix factorizations on the fibre is equivalent to the derived category of 2 points, so generically $(X_2, W)$ looks like a double cover of $\P^{n-1}$. More careful analysis at the degenerate points reveals that $\D(X_2, W)$ is actually a non-commutative resolution of a ramified double cover of $\P^{n-1}$. 

Our equivalence $\Psi_3$ is partially based on the techniques of this third step.
\end{rem}

\begin{rem}
It is reasonable to ask what happens if we vary the dimensions of $S$ and $V$, giving them dimensions $r$ and $d$ respectively, say, and correspondingly adapt the definitions of $\mfX$, $X_1$, $X_2$ and $W$. This affects the three steps as follows:

\newlength\templabelwidth
\settowidth{\templabelwidth}{$\Psi_1$: }
\begin{list}{$\Psi_{\arabic{enumi}}$: }{\usecounter{enumi} \setlength\itemsep{5pt} \setlength\itemindent{-\templabelwidth} \addtolength\itemindent{-\itemindentcorrection}}
\item The definition of the first Calabi--Yau $Y_1$ also adapts immediately, and the equivalence $\Psi_1$ continues to hold, as it is a consequence of a much more general theorem. Of course $d$ must be big enough compared to $r$ for $Y_1$ to be non-empty.
\item If we keep $r=2$ and $d$ odd then the correct definition of $\Br(X_2, W)$ is clear and the equivalence $\Psi_2$ generalizes immediately. If we move beyond these cases then there are obvious guesses as to how to proceed mathematically (particularly when $r=2$ and $d$ is even), but we encounter an apparent discrepancy with the physical results; see Remark~\ref{rem:windowsinotherdims}.

\item This step is the most delicate, and the only other case that we can handle completely is $r=2, \,d=5$, which recovers the derived equivalence between an elliptic curve and its dual. In the case $r=2,\, d=6$, we can recover most of Kuznetsov's result on Pfaffian cubic 4-folds \cite{KuznetsovHPDlines}, and for $r=2,\, d>7$ our construction suggests a possible homological projective dual for $\Gr(2,d)$.  See Remark~\ref{rem:pfaffianforotherdimensions} for more details.

\end{list}
\end{rem}

\begin{rem}
More recently Hori has provided a second physical derivation of the Pfaffian-Grassmannian equivalence, using a dual model \cite{Hori}; see also \cite{HoriKnapp}. It would be very interesting to find a mathematical interpretation of this duality.
\end{rem}

\opt{amsart}{
\begin{acks}
We thank Matt Ballard for helpful discussions, Ronen Plesser for helpful comments, and the anonymous referee for a thorough reading and many useful suggestions.  N.A. and W.D. are grateful for the hospitality of the Hausdorff Research Institute for Mathematics, Bonn.  W.D.~is grateful for the support of Iain Gordon and EPSRC grant no.~EP/G007632/1, and of the World Premier International Research Center Initiative (WPI Initiative), MEXT, Japan. W.D.~also thanks Paul Aspinwall for supporting a visit to Duke University under NSF grant no.~DMS-0905923.  E.S.~is supported by an Imperial College Junior Research Fellowship.
\end{acks}
}

\section{Categories of matrix factorizations}\label{sect:categories}

In this section we recall some general background on `global' matrix factorizations.

\subsection{Landau--Ginzburg B-models and curved dg-sheaves}

We make the following definition.

\begin{defn} A \emph{Landau--Ginzburg (or LG) B-model} consists of:
\begin{itemize}
\item A smooth $n$-dimensional scheme (or stack) $X$ over $\C$.
\item A choice of function $W\in \Gamma_X(\mathcal{O}_X)$ (the \quotes{superpotential}).
\item An action of $\C^*$ on $X$ (the  \quotes{R-charge}).
\end{itemize}
We denote the above copy of $\C^*$ by $\C^*_R$. We require that:
\begin{enumerate}
\item $W$ has weight (\quotes{R-charge}) equal to 2.
\item $-1\in \C^*_R$ acts trivially.
\end{enumerate}
\end{defn}

We let $(X,W)$ denote a Landau--Ginzburg B-model, suppressing the R-charge data from the notation. In affine patches, $\cO_X$ is a graded ring (graded by R-charge, and concentrated in even degree), and $W$ is a degree 2 element. Such a thing is sometimes called  a `curved algebra'; it is a very special case of a curved $A_\infty$-algebra.

\begin{eg}\label{eg:schemeisLG}
Any (smooth) scheme $X$ defines a LG B-model, by setting $W\equiv 0$  and letting $\C^*_R$ act trivially. This is an important special case.
\end{eg}

\begin{eg}\label{eg.KPbasiceg}
Let $X=\C^2_{x,p}$ and $W=xp$. We let $\C^*_R$ act with weight zero on $x$ and weight 2 on $p$. This is a LG B-model, and it's the basic example to which Kn\"orrer periodicity applies (see Section~\ref{sect:KPoverapoint}).
\end{eg}

\begin{eg} \label{eg:specifyingRcharge}
The example we care about in this paper is the linear Artin stack
$$\mfX = \quotstackbig{\Hom(S,V)\oplus \Hom( V, \wedge^2 S)}{\GL(S)}$$
 introduced in Section~\ref{sect:geometricGLSM}. We've already specified the superpotential $W$ \eqref{eq:W}, but we need to also specify the R-charge, which we do letting $\C^*_R$ act on $\Hom(V, \wedge^2 S)$ with weight 2, and on $\Hom(S,V)$ with weight 0. These data define a Landau--Ginzburg B-model.

We also care about the open substacks $X_1, X_2\subset \mfX$. These have superpotentials given by the restriction of $W$, and each one is $\C^*_R$-invariant, so they define LG B-models. 
\end{eg}

We now give the appropriate notion of a sheaf on an LG B-model.

\begin{defn} A \emph{curved dg-sheaf} on $(X,W)$ is a sheaf $\cE$ of $\cO_X$-modules, equivariant with respect to $\C^*_R$, equipped with an endomorphism $d_\cE \colon \cE \to \cE $ of R-charge~1 such that
$$(d_\cE)^2 = W \cdot \id_\cE.$$
\end{defn}

Note that, in affine patches, $(\cE, d_\cE)$ is simply a graded module equipped with a `curved differential'.

\begin{term}We will call $(\cE, d_\cE)$ \textit{coherent} (resp.\ \textit{quasi-coherent}) if the underlying sheaf $\cE$ is coherent (resp.\ quasi-coherent). If $\cE$ is actually a finite-rank vector bundle, we will call $(\cE, d_\cE)$ a \textit{matrix factorization}.
\end{term}

We are primarily interested in matrix factorizations and coherent curved dg-sheaves.

\begin{rem}Notice that because $-1\in\C^*_R$ acts trivially on $X$, any curved dg-sheaf splits into `even' and `odd' eigensheaves
$$\cE = \cE_{\text{even}} \oplus \cE_{\text{odd}} $$
and the differential $d_\cE$ exchanges the two. There is a weaker definition of LG B-model where we neglect the R-charge and keep only this (trivial) $\Z/2$ action; this results in a $\Z/2$-graded category, whereas with R-charge we can construct a $\Z$-graded category.\end{rem}

There is a $\C^*_R$-equivariant line bundle on $X$ associated to any character of $\C^*_R$, and we denote these line bundles by $\cO[k]$. For any curved dg-sheaf $\cE$, we can shift the equivariant structure by tensoring with $\cO[k]$, and we denote the result by  $\cE[k]$.

\begin{rem}
Suppose that $W=0$ and $\C^*_R$ acts trivially, as in Example~\ref{eg:schemeisLG}. Then a curved dg-sheaf is precisely a complex of $\cO_X$-modules, and a matrix factorization is a 
bounded complex of vector bundles. In this case the shift functor $[1]$ is the usual homological shift.
\end{rem}

The following is a useful source of examples of curved dg-sheaves.

\begin{eg}\label{eg:skyscrapers}
Suppose $Z\subset X$ is a ($\C^*_R$-invariant) subvariety lying inside the zero locus of $W$. Consider the skyscraper sheaf $\cE=\cO_Z$, equipped with the zero endomorphism $d_{\cE}=0$. This defines a curved dg-sheaf, concentrated in even degree.
\end{eg}

\subsection{Categories of curved dg-sheaves}\label{sect:catsofcurveddgsheaves}

Now we discuss the morphisms between curved dg-sheaves. Let $(\cE, d_\cE)$ and $(\cF, d_\cF)$ be curved dg-sheaves, and let
$$\hom_X(\cE, \cF)$$ 
denote the usual sheaf of $\cO_X$-module homomorphisms between the underlying sheaves $\cE$ and $\cF$. This sheaf is $\C^*_R$-equivariant, and carries a 
differential given by the commutator of $d_\cE$ and $d_\cF$, so it is a curved dg-sheaf on the LG B-model $(X, 0)$. Its global sections
$$\Gamma_X\hom_X(\cE, \cF) $$
form a complex of vector spaces, graded by R-charge. Consequently, we can try to build a dg-category whose objects are matrix factorizations, or coherent curved dg-sheaves. Of course it would be naive just to use the chain complexes above for morphisms; we have to do some more work to define the dg-category correctly. There are essentially two approaches:
\begin{enumerate}
\item Take as objects all matrix factorizations, and as morphisms the complexes
$$\RDerived\Gamma_X\hom_X(\cE, \cF) $$
where $\RDerived\Gamma_X$ is a suitable monoidal functor that computes derived global sections. We may for example use Dolbeault resolutions, or \u{C}ech resolutions with respect to some fixed $\C^*_R$-invariant affine cover of $X$, if one exists. We denote the resulting dg-category by $\Perf(X,W)$.
\end{enumerate}

This was the approach adopted in \cite{Segal}. It is fairly concrete, but it has the major disadvantage
that we can only use matrix factorizations as objects -- in the ordinary derived category $\D(X)$ it would be very frustrating if we could only use locally-free resolutions of coherent sheaves and never the sheaves themselves. Consequently it is helpful to have a second, more technical approach. This was developed by Orlov \cite{OrlovNonaffine} and Positselski \cite{Positselski}.

\begin{enumerate}\addtocounter{enumi}{1}

\item Let $\QCohdgnv(X,W)$ denote the dg-category of quasi-coherent curved dg-sheaves, with morphisms defined `naively' as above. It is easy to check that this category contains mapping cones, so if we have a bounded chain-complex of curved dg-sheaves
$$\cE_\bullet = \ldots\to \cE_0 \to \cE_1 \to \cE_2\to \ldots $$
we can form the totalization Tot$(\cE_\bullet)$, and this is a curved dg-sheaf. We define a curved dg-sheaf to be \textit{acyclic} if it is (homotopy equivalent to) the totalization of an exact sequence.  Then we define $\QCohdg(X,W)$ to be the quotient (as a dg-category) of $\QCohdgnv(X,W)$ by the full subcategory of acyclic objects. Finally, we define $\Perf(X,W)$ to be the full subcategory of $\QCohdg(X,W)$ consisting of objects which are homotopy-equivalent to matrix factorizations.
\end{enumerate}

Fortunately these two approaches define quasi-equivalent dg-categories; this was proven by Shipman \cite[Prop.~2.9]{Shipman} for the case that $X$ is a scheme, but the argument works for quotient stacks without modification (see also \cite[Prop.~2.11]{LP} for a similar statement without R-charge). Equally, the choice of functor $\RDerived\Gamma_X$ in the first construction is not important.  From the second construction, it is clear that $\Perf(X,W)$ is pre-triangulated, i.e.\ it contains mapping cones. The shift functor acts by shifting R-charge equivariance, i.e.\ tensoring with $\cO[1]$. 

We denote the homotopy category of $\Perf(X,W)$ by $\D(X,W)$; this is a triangulated category. We'll adopt the convention that the set of morphisms between two objects in this category is the graded vector space
$$\Hom^\bullet_{D^b(X,W)}(\cE, \cF)$$
given by all homology groups of the chain-complex $\Hom^\bullet_{\Perf(X,W)}(\cE, \cF)$, not just the zeroeth homology. For the case $W=0$, this means we are using $\Hom^\bullet_{D^b(X)}(\cE, \cF)$ to denote the graded vector space of all Ext groups between $\cE$ and $\cF$.

\begin{rem} Denoting the homotopy category of $\Perf(X,W)$ by $\D(X,W)$ is only appropriate when $X$ is smooth; in the singular case the latter notation should mean something different. In particular, in the special case that $W\equiv 0$ and the R-charge is trivial, $\Perf(X,W)$ is precisely the dg-category of perfect complexes on $X$, whose homotopy category coincides with $\D(X)$ if and only if $X$ is smooth. \end{rem}

\begin{rem}  In the rest of the paper we will consider various functors between categories of matrix factorizations, and we will write everything at the level of the homotopy categories. However it will be clear from our constructions that everything is actually well-defined at the level of dg-categories.\end{rem}

\subsection{Basic properties}\label{rem:basicproperties}
We list some other basic properties of $\D(X,W)$ for later reference.

\begin{enumerate}
\item  \label{itm:cf_mf} If $X$ is a scheme which admits a $\C^*_R$-equivariant ample line bundle, then every coherent curved dg-sheaf is equivalent to a matrix factorization, and hence defines an object in $\D(X, W)$ \cite[Lemma 2.12]{Shipman}. Presumably this fact is still true when $X$ is one of the stacks considered in this paper, but we shall not attempt to prove it, since whenever we encounter a coherent curved dg-sheaf we will be able to see explicitly that it is equivalent to a matrix factorization.

\item \label{itm:him} Let $\cE$ and $\cF$ be two curved dg-sheaves in $\Perf(X,W)$. We have discussed the `global derived morphisms'
$$\Hom^\bullet_{\Perf(X, W)}(\cE, \cF)$$
which is a chain-complex of vector spaces, but we will also need the sheaf of `local derived morphisms'. If $U\subset X$ is a ($\C^*_R$-invariant) affine open set, then  $\Hom^\bullet_{\Perf(U,X)}(\cE, \cF)$ is a dg-module over the graded algebra $\cO_U$, i.e.\ a curved dg-sheaf on $(U, 0)$.  Gluing these together over $X$ gives us a curved dg-sheaf on $(X,0)$, which we denote by
$$\RDerived\hom_X(\cE, \cF).$$
We have
$$\Hom^\bullet_{\Perf(X,W)}(\cE, \cF) = \RDerived\Gamma_X \RDerived\hom_X(\cE, \cF).$$
 In practice this sheaf is quite easy to compute: we do it by replacing $\cE$ with an equivalent matrix factorization $E$, and then
$$\RDerived\hom_X(\cE, \cF) = \hom_X(E,\cF). $$

\item \label{itm:support}  If $E$ and $F$ are matrix factorizations on an affine scheme then it is a basic observation that $\hom_X(E,F)$ is acyclic away from the critical locus of $W$, because multiplication by any partial derivative $\partial_i W$ is zero up to homotopy. Consequently, for any two curved dg-sheaves $\cE$ and $\cF$ the derived morphism sheaf $\RDerived\hom_X(\cE, \cF)$ is acyclic away from the critical locus, so its homology sheaves are supported (set-theoretically) at the critical locus.  So the whole category $\D(X,W)$ is in some sense supported on the critical locus of $W$; cf.\ \cite{OrlovFormal}.

\item \label{itm:dsing} Let $Z$ be the zero locus of $W$ and
$$\zeta\colon Z \into X $$
the inclusion. Extending Example~\ref{eg:skyscrapers}, any curved dg-sheaf on $(Z,0)$ pushes forward to give a curved dg-sheaf on $(X, W)$, so we have a functor
$$\zeta_*\colon \D(Z,0) \to \D(X,W). $$
(Note that $Z$ is typically singular so we must use a modified definition of $\D(Z,0)$ here.)

If we neglect R-charge, it is well-known (e.g.\ \cite{OrlovNonaffine}) that this functor is essentially surjective, and its kernel is the category of perfect complexes on $Z$. This gives an equivalent definition of $\D(X,W)$ as the `derived category of singularities'
$$D_{sg}(W) = \D(Z)/\Perf(Z). $$
Presumably this is still true if we include R-charge, but we shall not bother to check the full statement here. We just note the easy fact that $\zeta_*\cO_Z$ is equivalent to the matrix factorization
$$\begin{tikzcd} \cO[1] \arrow[transform canvas={yshift=.4ex}]{r}{W} & \cO \arrow[transform canvas={yshift=-0.4ex}]{l}{1} \end{tikzcd} $$
and this is contractible. It follows quickly that if $P^\bullet$ is any $\C^*_R$-equivariant perfect complex on $X$ then $\zeta_*\zeta^*P^\bullet$ is contractible in $\D(X,W)$.
\end{enumerate}

\section{Kn\"orrer periodicity and the Grassmannian side}\label{sect:knorrer}

One of the most important classical facts about matrix factorizations is Kn\"orrer periodicity \cite{Knorrer}. We will briefly discuss this phenomenon, and various modern formulations of it that have appeared in recent years \cite{OrlovKP, Isik, Shipman, Preygel}, and conclude by showing our first equivalence $\Psi_1$ in \ref{cor.psi1}.

\subsection{Kn\"orrer periodicity over a point} \label{sect:KPoverapoint}

Consider a LG B-model $X =  \C^2$ with the superpotential $W = x_1 x_2$, and let $Y$ be the subscheme of $X$ consisting of just the origin (we neglect R-charge for the moment). In its simplest form, Kn\"orrer periodicity states that we have an equivalence
$$\D(Y) \cong \D(X, W ). $$
\begin{rem}
Since $Y$ is the critical locus of $W$, this is a situation where we may take \S \ref{rem:basicproperties}(\ref{itm:support}) very literally.
\end{rem}
Finding such an equivalence is the same thing as finding a curved dg-sheaf $\cE$ on $(X,W)$ which generates the whole category, and satisfies
$$\Hom^\bullet_{\D(X, W)}(\cE, \cE) = \C.$$
Recall that this space of morphisms is a graded vector space, so implicit here is the statement that there are no morphisms in non-zero degree. Thus the object $\cE$ behaves, homologically,  like an isolated point.

 There are many possible choices for such an $\cE$; one is the skyscraper sheaf along the $x_2$-axis
$$\cE = \cO_{\{x_1=0\}}$$
with $d_{\cE}=0$ (this is an instance of Example~\ref{eg:skyscrapers}). Then we get an equivalence from $\D(Y)$ to $\D(X,W)$ by mapping $\cO_Y$ to $\cE$.

\begin{rem}  This choice of $\cE$ breaks the symmetry between $x_1$ and $x_2$. This is an important feature: there is a second choice where we let $\cE$ be the skyscraper sheaf on the $x_1$-axis, and this produces a different equivalence, differing from the first one by a shift. A related fact is that if we want to add R-charge to this construction then we can do it by letting $\C^*_R$ act with weight 2 on $x_1$ and weight 0 on $x_2$, or vice versa, but this also breaks the symmetry. \end{rem}

This basic version of Kn\"orrer periodicity can be generalized in various directions. Firstly, we may replace $X=\C^2$ with  $X=\C^{2n}$, and $W$ with a non-degenerate quadratic function, so the critical locus of $W$ is still the origin.  We replace the isotropic line $\{x_1=0\} \subset \C^2$ with a choice of maximally isotropic subspace $\maxisoGenl \subset \C^{2n}$. Then one can check that $\cE=\cO_\maxisoGenl$ is point-like, and generates $\D(X,W)$, so as above it gives us an equivalence between the derived category of a point and $\D(X,W)$.

\subsection{In families: first version}
\label{sect:familiesfirstversion}

Now we can try to formulate this construction in families. Most obviously we could choose $X$ to be the total space of an even-rank vector bundle 
$$\pi\colon X \to Y$$
and $W$ to be a fibrewise non-degenerate quadratic form on $X$. Suppose we can find a subbundle $\maxisoGenl\subset X$ which gives a maximally isotropic subspace in each fibre.  Then for each point $y\in Y$ we have a curved dg-sheaf $\cE_y = \cO_{\maxisoGenl_y}$ on the fibre over $y$, and these fit together into a family $\cE = \cO_\maxisoGenl$ on the whole space. We want to consider the functor whose Fourier--Mukai kernel is $\cE$, i.e.\ it sends each skyscraper sheaf $\cO_y\in \D(Y)$ to the corresponding $\cE_y\in \D(X,W)$, and sends the whole structure sheaf $\cO_Y$ to $\cE$. In other words, we consider the diagram
$$ Y \stackrel{\pi}{\longleftarrow}  \maxisoGenl \stackrel{\iota}{\longrightarrow} X $$
and the induced functors
$$\D(Y) \stackrel{\pi^*}{\longrightarrow} \D(\maxisoGenl) \stackrel{\iota_*}{\longrightarrow} \D(X, W). $$
It is proven in \cite[Thm.~9.1.7(ii)]{Preygel} that, given such a $\maxisoGenl$, the composition $\pi^*\iota_*$ gives us an equivalence between $\D(Y)$ and $\D(X,W)$.\footnote{The existence of such an $\maxisoGenl$ is quite a strong condition; see \cite[\S 4.3]{ASS} for some discussion of this point.} 

\begin{rem}\label{rem:embeddings}
In particular, $\pi^*\iota_*$ is fully faithful. We pause to discuss this point in a little more detail, since the reasoning used will be important  in Section~\ref{sect:pfaffian}.

The functor $\pi^*\iota_*$ is linear over the sheaf of functions on $Y$, so fully-faithfulness can be checked locally on $Y$. Moreover if we restrict to an affine neighbourhood in $Y$ then the derived category is generated by the structure sheaf, so locally we only need to check fully-faithfulness on the structure sheaf. Therefore it's enough to check that the endomorphisms of
$$ \pi^*\iota_* \cO_Y =\cE \in \D(X, W)$$
agree with the endomorphisms of $\cO_Y \in \D(Y)$, as a sheaf over $Y$, i.e.\ that
$$\pi_*\RDerived\hom_X(\cE, \cE) \cong \cO_Y. $$
This statement is equivalent to the fully-faithfulness of $\pi^*\iota_*$; in particular it obviously implies that
$$\Hom^\bullet_{\D(X,W)}(\cE_y, \cE_y) \cong \Hom^\bullet_{\D(Y)}(\cO_y, \cO_y)$$
for all points $y\in Y$. Informally at least the converse implication also holds: if we have a family of orthogonal objects $\cE_y$, and each one is `point-like' in this sense, then the resulting kernel $\cE$ must give a fully faithful functor. 
\end{rem}

\subsection{In families: second version}

There is a more general family version of Kn\"orrer periodicity, based on the observation that we don't actually need a projection $\pi\colon X \to Y$, only a projection $\pi\colon \maxisoGenl \to Y$.  Specifically, we consider the total space of a vector bundle 
$$\pi\colon X \to B$$
over some base $B$, and let 
$$Y\subset B$$
be the zero locus of some transverse section $f\in \Gamma_B(X^\vee)$. We can equip $X$ with the superpotential 
$$ W = fp$$
where $p$ denotes the tautological section of $\pi^*X$. Since $f$ is transverse, $Y$ is smooth and is exactly the critical locus of $W$. The normal bundle $\cN_{Y/X}$ to $Y$ carries a non-degenerate quadratic form given by the Hessian of $W$, and furthermore this bundle has a canonical maximally isotropic subbundle given by $\maxisoGenl=X|_Y$. So we should be able to get an equivalence between $\D(Y)$ and $\D(X,W)$ using the diagram
$$ Y \stackrel{\pi}{\longleftarrow}  X|_Y \stackrel{\iota}{\longrightarrow} X. $$
Note that there is a more-or-less canonical way to add R-charge to this construction, by letting $\C^*_R$ act trivially on $B$ and with weight 2 on the fibres of $X$.

\begin{thm}[{\cite[Thm.~3.4]{Shipman}}]\label{thm.Knorrer}  Consider a LG B-model $(X,W)$ of the form described above, with $\C^*_R$ acting fibrewise with weight 2. Assume that the base $B$ is a smooth quasi-projective variety. Then the composition
$$\D(Y) \stackrel{\pi^*}{\longrightarrow} \D(X|_Y) \stackrel{\iota_*}{\longrightarrow} \D(X, W) $$
is an equivalence.
\end{thm}
Similar theorems are proven in \cite{OrlovKP} and \cite{Isik}.  Note that $\C^*_R$ is acting trivially on $Y$, so $\D(Y)$ really does mean the usual derived category of $Y$.

\subsection{Grassmannian example}
Now consider the LG B-model $(X_1,W)$ discussed in Section~\ref{sect:geometricGLSM}, and described more precisely in Example~\ref{eg:specifyingRcharge}. This model is exactly of the form specified by the above theorem: $X_1$ is the total space of the vector bundle $\pi\colon\cO(-1)^{\oplus 7}\to \Gr(2,V)$, and the R-charge is acting trivially on the Grassmannian and with weight 2 on the fibres. Also the superpotential is $W=fp$, where 
$$ f = A\circ \wedge^2 x $$
is a transverse section of $\cO(1)^{\oplus 7}$ on $\Gr(2,V)$ and $p$ is the tautological section of $\pi^*\cO(-1)^{\oplus 7}$. The zero locus of $f$ is the Calabi--Yau 3-fold $Y_1$, and hence Theorem~\ref{thm.Knorrer} yields the following.

\begin{cor}\label{cor.psi1} $\D(Y_1)$ is equivalent to $\D(X_1, W)$. \end{cor}

This concludes our discussion of the first equivalence $\Psi_1$.

\section{Windows} \label{sect:windows}

In this section we will define the category $\Br(X_2, W)$ and the equivalence $\Psi_2$.

\subsection{Without the superpotential} \label{sect:without}

Let
$$X_1 \stackrel{\iota_1}{\into} \mfX \stackrel{\iota_2}{\hookleftarrow} X_2$$
be the three spaces considered in Section~\ref{sect:geometricGLSM}. For the purposes of this section we set the superpotential $W$ to zero, and take the $\C^*_R$ action to be trivial, so $\D(X_i)$ and $\D(\mfX)$ are the usual derived categories.

We are interested in the relationship between $\D(X_1)$ and $\D(X_2)$. If $X_1$ and $X_2$ were manifolds (or orbifolds) then we would expect them to be derived equivalent, since they are birational and Calabi--Yau. What should we expect in this situation?

Physically, we can reason as follows. Using Hori and Tong's construction, we know that the sigma models with targets $X_1$ and $X_2$ lie in the same K\"ahler moduli space of CFTs.\footnote{We gloss over the fact that these targets are non-compact.} Consequently the B-models associated to each space are the same. In particular, they have the same category of B-branes, and so we should have two equivalent categories
$$\Br(X_1) \cong \Br(X_2).$$

Since $X_1$ is a manifold, we know that the category of B-branes $\Br(X_1)$ is $\D(X_1)$. However, $X_2$ is an Artin stack. A \textquotedblleft sigma-model" whose target is an Artin stack is really a gauge theory, and understanding the category of B-branes in a gauge theory is much more difficult.  We will not attempt to address this general question; instead we will make an \emph{ad hoc} definition of the category $\Br(X_2)$, constructing a fully faithful embedding
$$\D(X_1) \into \D(X_2) $$
and defining $\Br(X_2)$ as the image of this embedding.  The main motivation for our definition is just that it gives something equivalent to $\D(X_1)$, but we will give some \emph{a posteriori} justification (see Remark~\ref{rem:branesonfibres}).

To construct the embedding we will use the technique of `grade-restriction', or `windows', introduced by the third-named author in \cite{Segal}. This was directly inspired by the physics paper \cite{HHP}, but was also based on a long history of mathematical ideas by Beilinson, Kawamata, Van den Bergh, etc. What we do is to find a subcategory
$$\cG \subset \D(\mfX) $$
such that the restriction functor $\iota_1^*\colon \cG \to \D(X_1)$ is an equivalence, and the other restriction functor $\iota_2^*\colon \cG \to \D(X_2)$ is fully faithful. In fact this technique has now been developed into an elegant general theory  \cite{HL,BFK} which can be applied immediately in this example to show that such a $\cG$ exists.  Unfortunately the description that this theory gives of the image of $\cG$ inside $\D(X_2)$ is not explicit enough for our purposes, so we take a more hands-on approach.
\pgap

Observe that any representation of $\GL(S) = \GL(2)$ determines a vector bundle on each of the spaces that we are considering.  We will be interested in the \quotes{rectangle} of representations
\begin{equation}\label{eq:rectangularcollection}\WindowSet{\Sym^l S^\vee \otimes (\det S^\vee)^m_{\InSpaceOf{,}{}}}{3}{7}.\end{equation}
The associated vector bundles on $\Gr(2,V)$ form a (Lefschetz) full strong exceptional collection by \cite[Thm.~4.1]{KuznetsovECs}.  Let $T_{l,m}$ denote the vector bundle $\Sym^l S^\vee(m)$ on~$\mfX$ associated to $\Sym^l S^\vee \otimes (\det S^\vee)^m$, and let
\begin{equation}\label{eq:window}\cG = \Window{T_{l,m}}{3}{7} \;\; \subset \;\; \D(\mfX)\end{equation}
be the subcategory generated by this set of vector bundles.\footnote{Here (and throughout the paper) we mean \quotes{generated} in the strong sense, by taking shifts and cones but \emph{not} direct summands -- that is, $\cG$ consists of those objects that have a finite resolution in terms of this set of bundles.}

\begin{prop}\label{prop:window}
The restriction functor 
$$\iota_1^*\colon \cG \to \D(X_1)$$
 is an equivalence, and the restriction functor 
$$\iota_2^*\colon \cG \to \D(X_2)$$
 is fully faithful.
\end{prop}
\noindent Consequently we obtain an embedding of $\D(X_1)$ into $\D(X_2)$, and its image is the subcategory generated by the vector bundles associated to the representations \eqref{eq:rectangularcollection}.  We define $\Br(X_2)$ to be this subcategory.

\begin{rem}
 The reason we're not using the general theory of \cite{HL,BFK} is that it fails to identify this explicit set of generating bundles for the category $\cG$. This is because Kuznetsov's exceptional collection does not fit with the `grade-restriction rules' for this GIT problem (for comparison, Kapranov's exceptional collection fits the Grassmannian side perfectly, but not the Pfaffian side). It would be interesting to find a natural derivation of Kuznetsov's exceptional collection via GIT.
\end{rem}

We split the proof of Proposition~\ref{prop:window} into four lemmas.

\begin{lem}\label{lem.fullyfaithful} Both $\iota_1^*$ and $\iota_2^*$ are fully faithful. 
\end{lem}
\begin{proof}
It is enough to check this statement on the generators of $\cG$. On $\mfX$, there are no higher $\Ext$'s between them: since they are vector bundles we have
\[ \Ext^p_\mfX(T_{l,m}, T_{l',m'}) \cong \RDerived^p \Gamma_\mfX (T_{l,m}^\vee \otimes T_{l',m'}) \]
and the functor of taking $\GL(S)$-invariants (i.e.\ global sections) is exact. Also, the $\Ext^0$'s between the generators will not change when we restrict to either $X_1$ or $X_2$. To see this note that the complements of both substacks have codimension at least 2, so by Hartogs' lemma the space of all sections of the bundle $T_{l,m}^\vee \otimes T_{l',m'}$ doesn't change after restriction, and therefore neither does the space of $\GL(S)$-invariant sections.

So we need only check that the generators don't acquire any higher $\Ext$'s after restriction, i.e.\ that
$$\Ext^{>0}_{X_i} \big(\iota_i^*T_{l\InSpaceOf{'}{},m},\; \iota_i^*T_{\otherl, \otherm}\big) = 0 $$
for all $l, \otherl\in [0,3)$ and $m, \otherm\in [0,7)$, for both $i=1$ and $i=2$.

For $i=1$ we use the projection formula applied to the projection
$$\stackProjection{1}\colon X_1= \Tot\big({\cO(-1)^{\oplus 7}}\big) \to \Gr(2,V)$$
to compute the cohomology of
\begin{align*} &\RHom_{X_1} \big(\iota_1^*T_{l\InSpaceOf{'}{},m},\; \iota_1^*T_{\otherl,\otherm} \big)   \\
&\hspace{20pt}\cong \RHom_{X_1}  \left( \stackProjection{1}^* \Sym^l S^\vee (m),  \; \stackProjection{1}^* \Sym^{\otherl}\! S^\vee (\otherm) \right)  \\
&\hspace{20pt}\cong \RHom_{\Gr(2,7)}  \left( \Sym^l S^\vee (m),  \; \stackProjection{1}{}_* \stackProjection{1}^* \left(\Sym^{\otherl}\! S^\vee (\otherm) \right)\right)  \\
&\hspace{20pt}\cong \RHom_{\Gr(2,7)}  \left( \Sym^l S^\vee (m),  \; \Sym^{\otherl}\! S^\vee (\otherm) \otimes \Sym^\bullet \cO(1)^{\oplus 7} \right). \end{align*}
Our claim now follows from the vanishing result used in \cite{KuznetsovECs}, which is stated below in Lemma~\ref{lem.Kuznetsov}, and a minor extension of it, given in Lemma~\ref{lem.Kuznetsov_extended}.

For $i=2$ we work similarly, using the fact that $X_2$ has a projection
$$ \stackProjection{2}\colon X_2=\Tot \big(S^{\vee \oplus 7}\big) \to \mathcal{P}$$
to an Artin stack $\mathcal{P}=\quotstackbig{\wedge^2\!S^{\oplus 7} - \{0\} }{\GL(S)}$. There is a map $\delta\colon \mathcal{P} \to \P^6$ induced by $\det\colon \GL(S) \to \C^*$, and forgetting the isotropy groups. Now working as above for $X_1$, and using that the functor $\delta_*$ is exact, we have
\begin{align*} &\RHom_{X_2} \big(\iota_2^*T_{l\InSpaceOf{'}{},m},\; \iota_2^*T_{\otherl, \otherm}\big)   \\
&\hspace{20pt}\cong \RHom_{\mathcal{P}}\left(\Sym^l S^\vee(m), \;\SymOtherl\!\, S^\vee(\otherm) \otimes \Sym^\bullet S^{\oplus 7}    \right) \\
&\hspace{20pt}\cong \RDerived\Gamma_{\P^6} \delta_* \left( \Sym^l S \otimes \SymOtherl\!\, S^\vee \otimes (\det S^\vee)^{\otherm-m} \otimes \Sym^\bullet S^{\oplus 7}    \right).   \end{align*}
Now using the Littlewood--Richardson rule \cite[\S A.1]{FultonHarris} we may decompose this last bundle into direct summands corresponding to irreducible representations of $\GL(S)$. The summands we obtain are Schur powers $\Schur^{\mu} S^\vee$ with $\mu \leq (\otherm, \otherm+\otherl)$, with the maximal $\mu$ occurring being the highest weight for the bundle $T_{\otherl,\otherm}$. Now we evaluate  $\delta_* (\Schur^{\mu} S^\vee)$. Every point of $\mathcal{P}$ has non-trivial stabilizer $\SL(S) \subset \GL(S)$, and $\Schur^{\mu} S$ has non-trivial $\SL(S)$-invariant vectors only if $\mu=(\nu,\nu)$. In this case $\Schur^{\mu} S^\vee \cong (\det S^\vee)^{\nu}$ and hence $\delta_* (\Schur^{\mu} S^\vee) \cong \cO_{\P^6}(-\nu)$. This has no higher cohomology as long as $\nu \leq 6$, and so we are done because $\nu \leq \otherm\leq 6$ by construction.\end{proof}

The following two lemmas are calculations used in Lemma~\ref{lem.fullyfaithful} above.

\begin{lem}[{\cite[Lem.~3.5]{KuznetsovECs}}]\label{lem.Kuznetsov} Let $\Gr=\Gr(2,V)$, with $\dim V=n$ odd. If $0 \leq l, \otherl \leq \frac{1}{2}n-1$ and $0 \leq k \leq n-1$ then
\[ \Ext_{\Gr}^p \big(\Sym^{l} S^\vee,\; \SymOtherl\!\, S^\vee(-k) \big) \cong
\begin{cases} \Sym^{\otherl-l} S^\vee & \text{if $l \leq \otherl$, $k=0$, $p=0$,} \\
0 & \text{otherwise.}
\end{cases} \]
\end{lem}
\begin{proof}This is a specialisation of the result of \cite[Lem.~3.5]{KuznetsovECs} to odd-dimensional $V$, as required in our case. The proof is combinatorial, using the Littlewood--Richardson rule to decompose a bundle on the Grassmannian into direct summands corresponding to irreducible representations of $\GL(S)$, and then the Borel--Bott--Weil theorem (as explained in \cite[\S 3]{KuznetsovECs}) to calculate their cohomology.\end{proof}

\begin{lem}\label{lem.Kuznetsov_extended} In the setting of Lemma~\ref{lem.Kuznetsov} above, but with $k<0$, we have
$$\Ext^{>0}_{\Gr} \big(\Sym^{l} S^\vee,\; \SymOtherl\!\, S^\vee(-k)\big) = 0.$$
\end{lem}
\begin{proof}It suffices to check that $\Sym^{l} S \otimes \SymOtherl\!\, S^\vee(-k)$ on $\Gr$ has no higher cohomology. Following the proof of \cite[Lem.~3.5]{KuznetsovECs} we have
\begin{equation*}\Sym^{l} S \otimes \SymOtherl\!\, S^\vee(-k) \;\cong\;\Sym^{l-1} S \otimes \Sym^{\otherl-1} S^\vee(-k)\;\oplus\;\Schur^{\otherl-k,-l-k} S^\vee, \end{equation*}
and so we may proceed inductively. We therefore need only check that the Schur power $\Schur^\alpha S^\vee$ has no higher cohomology on the Grassmannian $\Gr$ for $$\alpha = \big(\otherl-k,-l-k,0,\ldots,0\big).$$ The proof then follows by application of the Borel--Bott--Weil theorem, with the following two cases.

\emph{Case $k \leq -l$.} In this case $\alpha$ is a dominant weight, and hence there is no higher cohomology.

\emph{Case $-l < k < 0$.} Using $\rho$ to denote half of the sum of the positive roots of $\GL(n)$ as in \cite{KuznetsovECs}, we have that $$\alpha + \rho = \big(n+\otherl-k, n-l-k-1,n-2,n-3,\ldots,1\big).$$ Our assumptions give that $n-1 > n-l-k-1 > \frac{1}{2}n > 0,$ and hence the second entry in this weight coincides with one of the later ones. By the Borel--Bott--Weil prescription, it follows from this that no cohomology occurs in this case.

This completes the proof of the lemma.\end{proof}

The final stage in the proof of Proposition~\ref{prop:window} is the following.

\begin{lem}\label{lem:iota1essentiallysurjective} $\iota_1^*\colon\cG \to \D(X_1)$ is essentially surjective.
\end{lem}
\begin{proof}
This is the statement that the set of vector bundles on $X_1$ corresponding to the set \eqref{eq:rectangularcollection} of $\GL(S)$-representations generate the derived category $\D(X_1)$. This may be deduced from the fact that the corresponding set of vector bundles on $\Gr(2,V)$ generates the derived category $\D(\Gr(2,V))$ by Kuznetsov's result \cite[Thm.~4.1]{KuznetsovECs}, as follows.

First note that any coherent sheaf $\cE$ on $X_1$ extends to a coherent sheaf $\cE'$ on $\mfX$, and since $\mfX$ is smooth this extension $\cE'$ has a finite resolution by vector bundles. Furthermore, the only vector bundles which occur are the $T_{l,m}$ associated to $\GL(S)$-representations, as $\mfX$ is a quotient of a vector space by $\GL(S)$. Restricting this resolution via the inclusion $\iota_1\colon X_1 \into \mfX$ we obtain a finite resolution of $\cE$ on $X_1$ by the $\iota_1^*T_{l,m}$. We have that \[\iota_1^*T_{l,m} = \stackProjection{1}^* \left(\Sym^l S^\vee (m) \right),\] so that the $\iota_1^*T_{l,m}$ are pullbacks via  the projection $\stackProjection{1}\colon X_1 \to \Gr(2,V)$ of the bundles $\Sym^l S^\vee (m)$ on $\Gr(2,V)$. These latter bundles are themselves resolved by Kuznetsov's full exceptional collection corresponding to the set \eqref{eq:rectangularcollection} of $\GL(S)$-representations, and hence we deduce the result.
\end{proof}

This concludes the proof of Proposition~\ref{prop:window}.

\begin{rem}\label{rem:branesonfibres}
Recall that we're making an \emph{ad hoc} definition of the category of B-branes on $X_2$ as
$$\Br(X_2) := \iota_2^* \cG = \Window{\iota_2^* T_{l,m}}{3}{7} \;\;\subset\;\; \D(X_2).$$
Let's explain why this definition is not totally unreasonable. We have that $X_2$ is a bundle over $\P^6$, with fibres 
$$\stackFibre = \quotstackbig{\Hom(S, V)}{\SL(S)},$$
and so we should expect $\Br(X_2)$ to be some kind of product of $\Br(\P^6) = \D(\P^6)$ with some category $\Br(\stackFibre)$ of B-branes on the fibres. The derived category of $\P^6$ is generated by the Beilinson exceptional collection
$$\setcondsbig{\iota_2^* T_{0,m}=\cO(m)}{m\in[0,7)},$$
so what we're implicitly doing is declaring that
$$\Br(\stackFibre) = \left\langle \cO,\; S^\vee,\; \Sym^2 S^\vee\right\rangle \;\;\subset\;\; \D(\stackFibre).$$
We don't have a justification for this definition either, but it does satisfy
 $$\rk K_0(\Br(\stackFibre))=3$$ 
which matches Hori--Tong's calculation of the Witten index for the gauge theory described by $\stackFibre$, see \cite[Table 1]{HT}.
\end{rem}

\begin{rem}\label{rem:windowsinotherdims}
Let's briefly discuss how one might adapt this argument if we were to vary the dimensions of $S$ and $V$, making them $r$ and $d$ respectively. The general theory of \cite{HL, BFK} still gives us an embedding of $\D(X_1)$ into $\D(X_2)$, but as before it tells us very little about the image. So we should ask to what extent our more explicit methods can be adapted.

 If we keep $r=2$ and $d$ odd then everything works essentially verbatim, using the rectangular window
$$\WindowSet{ \Sym^l S^\vee \otimes (\det S^\vee)^m}{\tfrac{1}{2}(d-1)}{d}.$$
Now let's keep $r=2$, but make $d$ even. Something goes wrong even at the crude heuristic level of Remark~\ref{rem:branesonfibres}, because now $d$ does not divide $\binom{d}{r}$. Mathematically, it seems sensible to declare that $\Br(X_2)$ is the subcategory generated by the rectangle
$$\WindowSet{ \Sym^l S^\vee \otimes (\det S^\vee)^m }{\tfrac{1}{2} d}{d}.$$
If we delete $\tfrac12 d$ bundles from the corner of this rectangle then we get  Kuznetsov's (non-rectangular) Lefschetz exceptional collection on $\Gr(2, d)$, and we see that we obtain an embedding of $\D(X_1)$ into $\Br(X_2)$, rather than an equivalence. This definition allows us to recover a result of Kuznetsov in the case $r=2$ and $d=6$ (see Remark~\ref{rem:pfaffianforotherdimensions}). Unfortunately, this definition does not appear to be compatible with the results of \cite{HT}. It suggests that the category of B-branes on the fibre $\stackFibre$ should be generated by 
$$ \setconds{\Sym^l S^\vee}{l\in \left[0,\tfrac12 d\right)}$$
but Hori--Tong calculate the Witten index of the corresponding gauge theory to be $(\tfrac12 d-1)$, not $\tfrac12 d$. It would be very interesting to understand why these two approaches seem to give different answers.

If we make $r>2$ then we can presumably make some mathematical progress using Fonarev's Lefschetz exceptional collections on $\Gr(r, d)$ \cite{Fonarev}, but the discrepancy with Hori--Tong's calculation becomes even worse.
\end{rem}

\subsection{With the superpotential}\label{sect:withthesuperpotential}

We'll now explain how to modify the constructions of the previous section when we add in the superpotential $W$, and the non-trivial R-charge described in Example~\ref{eg:specifyingRcharge}. Specifically, we'll show that we have an embedding
$$\D(X_1, W)  \into \D(X_2, W).$$
The construction of this embedding follows closely our construction of the embedding $\D(X_1)\into \D(X_2)$. Suppose we have some matrix factorization $E\in \D(\mfX, W)$ on the ambient Artin stack. The underlying vector bundle of $E$  must be a direct sum of shifts of the bundles $T_{l,m}$, since these are the only vector bundles on $\mfX$. To define the analogue of the window $\cG$, we just restrict which vector bundles $T_{l,m}$ we are allowed to use. Namely, we define
$$\cG_W\subset \D(\mfX, W)$$
to be the full subcategory whose objects are (homotopy equivalent to) matrix factorizations whose underlying vector bundles are direct sums of shifts of the vector bundles $T_{l,m}$, where $l\in [0,3)$ and $m\in [0,7)$, as in \eqref{eq:window}. We then have the following.

\begin{prop} The restriction functor
$$\iota_1^*\colon \cG_W \to \D(X_1, W) $$
is an equivalence, and  the restriction functor
$$\iota_2^*\colon \cG_W \to \D(X_2, W) $$
is fully faithful.
\end{prop}
\begin{proof}
This follows from Proposition~\ref{prop:window}, using the arguments from \cite[\S 3.1]{Segal}. Fully-faithfulness is straightforward; we can use the proof of [\textit{ibid.}, Lem.~3.4] verbatim. The key point is that morphisms on any Landau--Ginzburg model $(X,W)$ can be computed, via a spectral sequence, from morphisms on the model $(X,0)$.

The essential surjectivity of $\iota_1^*$ follows from Lemma~\ref{lem:resolvingintowindows} below, since we proved in Lemma~\ref{lem:iota1essentiallysurjective} that  any sheaf on $X_1$ can be resolved by 
vector bundles from the set \eqref{eq:rectangularcollection}, and this resolution can evidently be chosen to be $\C^*_R$-equivariant.
\end{proof}

\begin{lem}\label{lem:resolvingintowindows}
Let $(X,W)$ be a LG B-model, and let $E_0,\ldots,E_k$ be a collection of $\C^*_R$-equivariant vector bundles on $X$ such that
$$\Ext^{> 0}_X (E_i, E_j) = 0, \;\;\;\;\forall i, j $$
 in the ordinary derived category of $X$ (i.e.\ ignoring the R-charge grading). Now let
$$(\cE, d_{\cE})\in \D(X, W)$$
be an object such that the underlying sheaf $\cE$ has a finite $\C^*_R$-equivariant resolution by copies of shifts of the bundles $E_i$. Then $(\cE, d_{\cE})$ is equivalent to a matrix factorization whose underlying vector bundle is a direct sum of copies of shifts of the $E_i$.
\end{lem}
\begin{proof}
This is proved in \cite[proof of Lem.~3.6]{Segal}. It's shown there that it's possible to perturb the differential in the resolution of $\cE$ until it becomes a matrix factorization for $W$ which is equivalent to $(\cE, d_{\cE})$.\footnote{The proof in that paper is stated for the case that $\cE$ is a vector bundle, but it works for sheaves without modification. The argument is also independent of which dg model we choose for $\Perf(X,W)$.}
\end{proof}

We define the category
$$\Br(X_2, W) \subset \D(X_2,W) $$
to be the image of $\cG_W$ under $\iota_2^*$, and we claim that this is the correct category of B-branes for the LG model $(X_2,W)$.

This concludes our discussion of the second equivalence $\Psi_2$.

\section{The Pfaffian side}\label{sect:pfaffian}

In this final section we complete our proof that $\D(Y_1) \cong \D(Y_2)$ by establishing the equivalence $\Psi_3$. To do this we construct an embedding
$$\D(Y_2) \into \D(X_2, W) $$
whose image is the subcategory $\Br(X_2, W)$ defined in the previous section. 

Recall that $X_2$ is the Artin stack 
\[ X_2 = \quotstackBig{ \setcondsbig{(x,p) \in \Hom(S,V) \oplus \Hom(V,\wedge^2 S)}{p \ne 0}}{\GL(S)} \]
and that it is equipped with the superpotential
\[W(x,p) = p \circ A \circ \wedge^2 x, \]
where $A\colon\wedge^2\!V \to V$ is a surjection satisfying Assumption \ref{assm:generic} that we've fixed throughout the paper. For this section, we'll let $\pi$ denote the projection
\[ \setlength \arraycolsep {2pt}
\begin{array}{rccl}
\pi\colon & X_2 & \to & \P\Hom(V,\wedge^2 S) \cong \P^6. \\
& (x,p) & \mapsto & [p] 
\end{array} \]
This makes $X_2$ into a Zariski locally-trivial bundle of stacks with fibre
\begin{equation} \label{eq:stackfibre}
\stackFibre := \quotstackbig{\Hom(S,V)}{\SL(S)}.
\end{equation}
To see this, observe that the preimage of a standard affine chart $\A^6\subset \P^6$ is the stack
\begin{align*}
X_2|_{\A^6}
&= \quotstackbig{\Hom(S,V)\times \C^*\times \A^6}{\GL(S)}   \\
&\cong \quotstackbig{\Hom(S,V)}{\SL(S)}\times \A^6.
\end{align*}
For another point-of-view, we can consider $X_2$ as a vector bundle
$$q_2: X_2 \to \mathcal{P} $$
over the stack
$$\mathcal{P}=\quotstackbig{\wedge^2\!S^{\oplus 7} - \{0\} }{\GL(S)}$$
(this was mentioned briefly in the proof of Lemma \ref{lem.fullyfaithful}). Then we can factor $\pi$ as $\delta\circ q_2$, where $\delta$ is the forgetful map
$$\delta: \mathcal{P} \to \P^6$$
sending $\mathcal{P}$ to its underlying scheme. The map $\delta$ is a Zariski locally-trivial bundle of stacks with fibre $B\SL_2$.

Note that the $\C^*_R$ action on $X_2$ preserves each fibre of $\pi$, and if we write a fibre using the atlas \eqref{eq:stackfibre} then the action on $\stackFibre$ is just dilation (with weight 1). This is because letting $\C^*_R$ act with weight 2 on $p$ is equivalent to letting it act with weight 1 on $x$, up to the action of $\GL(S)$.

\subsection{Heuristics and strategy} \label{sect:heuristicargument}
Fix a point $[p]\in \P^6$. On the fibre $ X_2|_{[p]} \cong \stackFibre $ over this point the superpotential is a quadratic form:
\[ W_p(x) := p \circ A \circ \wedge^2 x. \]
Since the $\C^*_R$ action preserves the fibre, the pair $(\stackFibre, W_p) $ is a LG B-model in its own right.  If the quadratic form $W_p$ were non-degenerate then our discussion of Kn\"orrer periodicity in Section~\ref{sect:familiesfirstversion} would lead us to study $\SL(S)$-invariant, maximally isotropic subspaces
\[ \maxiso_p \subset \Hom(S,V) \]
in order to understand $\D(\mathfrak{F},W_p)$. In fact $W_p$ is degenerate, but previous experience \cite{ASS} suggests that this is still a sensible thing to do.

To ensure $\SL(S)$-invariance, we need to take $M_p=\Hom(S, L_p)$, where $L_p \subset V$ is maximally isotropic for the 2-form
\[ \twoform{p} := p \circ A \]
on $V$.  The rank of this 2-form is 6 for a generic $[p]$, and it drops to 4 precisely when $[p] \in Y_2$. Since $A$ is generic, it never drops to 2.  Thus if $[p] \notin Y_2$ then a maximal $L_p$ has dimension 4 and a maximal $M_p$ dimension 8, but if $[p] \in Y_2$ then $\dim L_p$ jumps up to 5 and $\dim M_p$ to 10.

In fact, we will restrict attention to maximally isotropics $M_p$ for $[p]\in Y_2$, for the reasons we now explain. Our results from the previous section (see in particular Remark~\ref{rem:branesonfibres}) suggest that we should focus on the `window' subcategory
$$\Br(\mathfrak{F}, W_p)\subset \D(\mathfrak{F}, W_p)$$
consisting of (objects isomorphic to) matrix factorizations built only out of the three vector bundles $\cO$, $S$ and $\Sym^2 S$. This category is, in some sense, the fibre of the category $\Br(X_2, W)$ at the point $[p]$. Consequently we only care about those maximally isotropics $M_p$ that define objects in the subcategory $\Br(\mathfrak{F}, W_p)$.

The sheaf $\cO_{M_p}$ has a Koszul resolution with underlying vector bundle
\begin{equation} \label{eqn.koszul_for_supp_argument}
\wedge^\bullet \big(\Hom(S, V/L_p)^\vee\big).
\end{equation}
Perturbing the Koszul differential as in Lemma~\ref{lem:resolvingintowindows}, we find that $\cO_{M_p}\in\D(\stackFibre, W_p)$ is equivalent to a matrix factorization with this same underlying vector bundle. Then we use the formula for the exterior algebra of a tensor product \cite[Cor.~2.3.3]{Weyman} to find that the representations of $\SL(S)$ occurring in \eqref{eqn.koszul_for_supp_argument} are $\Sym^t S$, for 
$$0 \le t \le \dim(V/L_p).$$
To get $\cO_{M_p} \in \Br(\cF, W_p)$ it appears that we need to have $\dim(V/L_p) = 2$, and hence $[p] \in Y_2$. So if we believe this heuristic argument, the category $\Br(X_2, W)$ is concentrated over the Pfaffian locus $Y_2$.
\opt{amsart}{\pgap}
 
In the spirit of Section~\ref{sect:familiesfirstversion}, a continuous choice of $L_p$ for all $[p] \in Y_2$ will give us a functor $\D(Y_2) \to \Br(X_2,W)$ sending $\cO_{[p]}$ to $\cO_{M_p}$. We claim that this functor is in fact fully faithful. This is essentially equivalent (see Remark~\ref{rem:embeddings}) to the claim that each object $\cO_{M_p}$ behaves like the point sheaf $\cO_{[p]}$, i.e.
\[\Hom^\bullet_{\D(X_2,W)}(\cO_{M_p},\cO_{M_p}) \cong \Hom^\bullet_{\D(Y_2)}(\cO_{[p]},\cO_{[p]}),\]
or alternatively to the claim that the whole family $\cO_M$ behaves like the structure sheaf $\cO_{Y_2}$, i.e.
\[\pi_* \RDerived\hom_{X_2}(\cO_M, \cO_M) \cong \cO_{Y_2}. \]
A suitable version of this claim will be proved in Proposition~\ref{prop:maxisocalculation}, but let's briefly discuss why it is true. If each quadratic form $W_p$ were non-degenerate then it would be standard Kn\"orrer periodicity, and each object $\cO_{M_p}$ would be point-like in the fibrewise directions. However since $W_p$ is degenerate this is not true: viewed as an object on $(\mathfrak{F}, W_p)$ the curved dg-sheaf $\cO_{M_p}$ is not point-like -- it in fact looks like the skyscraper sheaf along the kernel of $W_p$. Fortunately this calculation is misleading, because if we view $\cO_{M_p}$ as an object on $(X_2, W)$ then we must also take account of the derivatives of $W$ in the directions transverse to the fibre. As we shall see, these transverse directions exactly cancel the degenerate directions of $W_p$, leaving a suitably point-like object. 

Next we face another issue, which is that the spaces $L_p$, and hence $M_p$, can be chosen locally on $Y_2$ but not globally.  One approach to overcoming this would be to take local choices and glue them to give a global embedding. Instead we replace each $\cO_{M_p}$ with an equivalent object $\cO_{\Gamma_p} \in \D(\stackFibre, W_p)$ which involves no choices and thus is easy to globalize to a family $\Gamma$. We define $\Gamma$ in Definition~\ref{defn:Gamma} and show in Proposition~\ref{prop.rank1vsmaximallyisotropic} that $\cO_{M_p}$ is equivalent to the new object $\cO_{\Gamma_p}$.

In Section~\ref{sec:finale} we fill in the final details that $\Gamma$ gives us an embedding $\D(Y_2) \to \D(X_2, W)$ whose image is $\Br(X_2,W)$.  We conclude with some remarks on varying the dimensions of $S$ and $V$, and on homological projective duality.

\subsection{The critical locus}

We start by analyzing the critical locus of $W$ on $X_2$. This means we take the critical locus of $W$ on the atlas
$$\Hom(S,V)\times \left(\Hom(V, \wedge^2 S) - \{0\}\right)$$
and form the stack quotient of it by $\GL(S)$.\footnote{The result is independent of our choice of atlas. If $[Z/G]$ is a quotient stack with $Z$ smooth, and $W$ is a $G$-invariant function on $Z$, then at any point $z\in Z$ the derivative $dW|_z$ defines a closed element of the cotangent complex  $L_z = [T^\vee Z \to \mathfrak{g}^\vee]$. We're considering the substack where this element is zero, and this is invariant since the cotangent complex is an invariant of the atlas up to quasi-isomorphism.}

\begin{prop}\label{prop.Pfcritical} 
Let $x \in \Hom(S,V)$ and $p \in \Hom(V,\wedge^2 S)-\{0\}$.  Then $(x,p)$ is a critical point of $W$ if and only if $\Im(x) \subset \ker \twoform{p}$ and $\rk(x) \le 1$.
\end{prop}
\begin{proof}
In the $x$-directions $W$ is a quadratic form, so its derivatives vanish exactly along its kernel, which is $\Hom(S, \ker \twoform{p})$.  In the $p$-directions $W$ is linear, so its derivatives vanish exactly when $W(x,q) = 0$ for all $q \in \Hom(V,\wedge^2 S)$.  Thus $(x,p)$ is a critical point of $W$ if and only if $\Im(x)$ is contained in $\ker \twoform{p}$ and is isotropic for all $\twoform{q}$ as $q$ varies over $\Hom(V,\wedge^2 S)$.  Now we need only argue that these imply $\rk(x) \le 1$.  If $\rk(\twoform{p}) = 6$ then $\dim(\ker \twoform{p}) = 1$, so $\rk(x) \le 1$ already, but if $\rk(\twoform{p}) = 4$ we need a further argument.

Consider the locus of $\omega \in \Hom(\wedge^2 V, \wedge^2 S)$ for which $\omega$ has rank 4 as a 2-form on $V$.  By \cite[Ex.~20.5]{Harris}, a line $\omega + t \xi$ is tangent to this locus if and only if $\ker \omega$ is isotropic for $\xi$; that is, the tangent space to this locus is the kernel of the natural map
\[ 
\setlength \arraycolsep {2pt}
\begin{array}{ccc}
\Hom(\wedge^2 V, \wedge^2 S) & \to & \Hom(\wedge^2 \ker \omega, \wedge^2 S). \\
\xi & \mapsto & \xi|_{\ker \omega}
\end{array} \]
Thus the normal space to this locus in $\Hom(\wedge^2 V, \wedge^2 S)$ embeds into \opt{amsart}{\[\Hom(\wedge^2 \ker \omega, \wedge^2 S),\]}\opt{alggeom}{$\Hom(\wedge^2 \ker \omega, \wedge^2 S),$} and since both have dimension 3 they are isomorphic.

Now by Assumption \ref{assm:generic}, $A$ gives an embedding $\Hom(V, \wedge^2 S) \into \Hom(\wedge^2 V, \wedge^2 S)$ which is transverse to the rank-4 locus, so the normal space to the cone on $Y_2$ under this embedding is identified with $\Hom(\wedge^2 \ker \twoform{p}, \wedge^2 S)$ in the same way.  In particular, for every 2-form $\eta$ on $\ker \twoform{p}$ there is a $q \in \Hom(V,\wedge^2 S)$ such that $\twoform{q}|_{\ker \twoform{p}} = \eta$.  Now if $\Im(x) \subset \ker \twoform{p}$ were 2-dimensional there would be an $\eta$ for which it was not isotropic, hence a $q$ such that $\Im(x)$ was not isotropic for $\twoform{q}$, so $(x,p)$ would not be a critical point of $W$.  Thus if $(x,p)$ is a critical point of $W$ then $\rk(x) \le 1$ as claimed.
\end{proof}

We now focus on the part of the critical locus that lies over the Pfaffian Calabi--Yau $Y_2$. Let 
$$K \subset \cO_{Y_2} \otimes V$$
 be the rank-3 bundle whose fibre over $[p]\in Y_2$ is 
$$K_p := \ker \twoform{p} \subset V.$$
In the proof of the previous proposition we saw that $dW$ induces an isomorphism
\begin{equation} \label{eqn.transversederivativesofW}
dW \colon \cN_{Y_2/\P^6} \to \Hom(\wedge^2 K, \wedge^2 S)
\end{equation}
of vector bundles over $Y_2$. Notice that although $S$ is not really a vector bundle on $\P^6$ (it's a vector bundle on the stack $\mathcal{P}$), its determinant $\wedge^2 S$ really is a legitimate  line bundle on $\P^6$ -- in fact it's $\cO_{\P^6}(1)$.

We abuse notation slightly and let
$$\Hom(S,K) \;\subset\; \Hom(S,V)\times \left(\Hom(V, \wedge^2 S) - \{0\}\right)$$
denote the subvariety 
$$\Hom(S,K) = \setconds{(x,p) \in X_2}{[p]\in Y_2, \; x\in \Hom(S, K_p)}.$$
This is a vector bundle over the punctured affine cone over $Y_2$, and \opt{amsart}{\[\quotstack{\Hom(S, K)}{\GL(S)}\]}\opt{alggeom}{$\quotstack{\Hom(S, K)}{\GL(S)}$} is a substack of $X_2$ whose fibre over a point $[p] \in Y_2$ is
\[ \quotstack{\Hom(S,K_p)}{\SL(S)} \subset \quotstack{\Hom(S, V)}{\SL(S)} = \stackFibre. \]
Alternatively, we may view $\quotstack{\Hom(S, K)}{\GL(S)}$  as a vector bundle over the stack $\mathcal{P}$.

\begin{lem} \label{lem.coarsemodulispaces}
The underlying scheme of $\quotstack{\Hom(S, K)}{\GL(S)}$ is the total space of the vector bundle $\Hom(\wedge^2 S, \wedge^2 K)$ over $Y_2$. The underlying scheme of  $\Crit(W)|_{Y_2}$  is $Y_2$.
\end{lem}
The `underlying scheme' of a stack is the universal scheme that receives a map from that stack; for a quotient stack this is simply the scheme-theoretic quotient.

\begin{proof}
As we just mentioned, the fibre of $\quotstack{\Hom(S, K)}{\GL(S)}$ over a point $[p]\in Y_2$ is the stack $\quotstackbig{\Hom(S,K_p)}{\SL(S)}$.  The scheme underlying this is the scheme-theoretic quotient
$$\Hom(S, K_p)\, /\, \SL(S) = \operatorname{Spec} \left( \cO_{\Hom(S,K_p)}\right)^{\SL(S)}. $$
By \cite[\S 8.4]{KraftProcesi} we have a closed embedding
$$ \Hom(S, K_p)\, /\, \SL(S) \into \Hom(\wedge^2 S, \wedge^2 K_p)$$
which is an isomorphism since both spaces have dimension~3. The first statement of the lemma follows immediately.

By Proposition \ref{prop.Pfcritical}, $\Crit(W)|_{Y_2}$ is the substack of $\quotstack{\Hom(S, K)}{\GL(S)}$ where $x$ has rank $1$. The image of this substack in the underlying scheme is the zero section.
\end{proof}

One can argue similarly that the underlying scheme of the whole of $\Crit(W)$ is $\P^6$, but we shall not use this fact.

\subsection{Point-like objects from maximally isotropic subspaces}
\label{sect:maxisocalculation}

We now show that maximally isotropic subspaces give point-like objects, as we outlined in \S \ref{sect:heuristicargument}.  Recall that, roughly, we want to find a family $L_p$ of maximally isotropic subspaces for the family of (rank 4) 2-forms $\omega_p$ over $Y_2$. Then we're going to look at the corresponding maximally isotropic subspaces $M_p=\Hom(S,L_p)$ for the associated family of quadratic forms. We don't know that we have such a family $L$ globally on $Y_2$, but we can find one Zariski locally (see Remark \ref{rem.localisotropicbundle} below).

 Let's begin by stating this local data precisely. Suppose we have an affine open set $\opensetinP \subset \P^6$ such that over the corresponding open set $\opensetinY := Y_2 \cap \opensetinP$ in $Y_2$ we can find a bundle
$$L \subset  \cO_{\opensetinY} \otimes V$$
of maximally isotropic subspaces for the family of 2-forms $\twoform{p}$.  As we did for $\Hom(S,K)$ above, let us use the notation
$$M := \Hom(S,L)\;\subset\;\Hom(S,V)\times \left(\Hom(V, \wedge^2 S) - \{0\}\right)$$
to denote the subvariety
$$M := \Hom(S,L) = \setconds{(x,p)}{[p]\in \opensetinY, \; x\in \Hom(S, L_p)}.$$
Notice that $M$ is preserved by both $\GL(S)$ and $\C^*_R$. 

The variety $M$ is a vector bundle over the punctured affine cone on $Y_2$, and its fibres are maximally isotropic subspaces $M_p$ for the family of quadratic forms $W_p$. The stack $\quotstack{M}{\GL(S)}$ is a vector bundle over the stack $\mathcal{P}|_{Y_2}$, which is like $Y_2$ but with $\SL_2$ isotropy groups at each point.

Since $M$ lies in the zero locus of $W$, the skyscraper sheaf $\cO_M$ is a curved dg-sheaf on the LG B-model $(X_2|_{\opensetinP}, W)$. For each point $[p]\in Y_2$, it restricts to give a curved dg-sheaf $\cO_{M_p}$  on the fibre $(\cF, W_p)$. As discussed in \S \ref{sect:heuristicargument}, we claim that these objects are `point-like' in the sense that they behave like the skyscraper sheaves $\cO_{[p]}\in \D(Y_2)$. This follows immediately from a slightly stronger claim, which is that the skyscraper sheaf $\cO_M$ along the whole family behaves like the structure sheaf on $\cO_{Y_2}$. This claim is our next proposition.

Recall that $\pi\colon X_2 \to \P^6$ is the projection sending $(x,p)$ to $[p]\in \P^6$. In our local situation, it is a map $\pi\colon X_2|_U \to U$.

\begin{prop}\label{prop:maxisocalculation} Suppose we have $U$, $Y'$, $L$ and $M$ as above.  Then the natural map
\[\cO_{\opensetinY} \To \pi_* \RDerived\hom_{X_2|_U}(\cO_\maxiso, \cO_\maxiso) \]
is a quasi-isomorphism.
\end{prop}
\begin{proof}
The natural map in question is the composition of the natural map $\cO_{\opensetinY}\to \pi_*\cO_M$ with the pushdown of the identity for $\cO_M$.  This is necessarily non-zero everywhere, so it's enough to prove that $\pi_* \RDerived\hom_{X_2|_U}(\cO_\maxiso, \cO_\maxiso)$ is quasi-isomorphic to $\cO_{\opensetinY}$.

The curved dg-sheaf $\cO_M$ is the skyscraper sheaf along a smooth subvariety lying in the zero locus of $W$.  For a curved dg-sheaf of this form, it's easy to show that\footnote{We neglect some shifts in R-charge which will be irrelevant.}
\begin{equation} \label{eq:perturbed_normal_polyvector}
\RDerived\hom_{X_2|_U}(\cO_\maxiso, \cO_\maxiso) \cong (\wedge^\bullet \cN_{\maxiso /X_2}, \, dW);
\end{equation}
see for example \cite[\S A.4]{ASS}.  So we take the sheaf of normal polyvector fields (which would be the correct answer if $W$ were zero) and perturb it by contracting with the section
$$dW \colon \cO_\maxiso \to \cN^\vee_{\maxiso / X_2}, $$
which is well-defined since $W$ vanishes along $\maxiso$.  This is not a transverse section, in the sense that its intersection with the zero section is not transverse. However we will split it into two pieces, one of which is transverse and the other of which we analyzed earlier.

Since $\maxiso$ is a vector bundle over (the punctured affine cone over) $\opensetinY$, we have a short exact sequence
\begin{equation} \label{eqn.conormal_seq}
0 \to \pi^* \cN^\vee_{\opensetinY / \opensetinP} \to \cN^\vee_{\maxiso / X_2|_\opensetinP} \to \Hom(S, V/L)^\vee \to 0.
\end{equation}
Since the open set $\opensetinP$ is affine, the total space of $\maxiso$ is also affine, so the sequence \eqref{eqn.conormal_seq} splits:
\[ \cN^\vee_{\maxiso / X_2} \cong \pi^* \cN^\vee_{\opensetinY / U} \oplus \Hom(S, V/L)^\vee. \]
Write $dW = (dW)_1 \oplus (dW)_2$ with respect to this splitting.  Then the right-hand side of \eqref{eq:perturbed_normal_polyvector} is a tensor product of the Koszul complexes associated to $(dW)_1$ and $(dW)_2$.

The section $(dW)_2$ consists of the fibre-wise derivatives of the family of quadratic forms $W_p$, so as in the proof of Proposition~\ref{prop.Pfcritical} it vanishes exactly along the kernel $\Hom(S, K)$ of the family of quadratic forms.  Since $\dim \Hom(S,K) = \dim \Hom(S,L) - \dim \Hom(S,V/L)^\vee$ we see that $(dW)_2$ is transverse, so the associated Koszul complex is exact, and we may replace it with $\cO_{\Hom(S, K)}$.

Thus $\RDerived\hom_{X_2|_U}(\cO_\maxiso, \cO_\maxiso)$ is quasi-isomorphic to the Koszul complex of the section
\[ (dW)_1 \colon \cO_{\Hom(S,K)} \to \pi^* \cN^\vee_{\opensetinY / U} \]
on the total space of $\Hom(S,K)$ over $\opensetinY$.  This section is not transverse, but what we actually care about is $\pi_*\RDerived\hom_{X_2|_U}(\cO_\maxiso,  \cO_\maxiso) $, which we can compute by first pushing down to the underlying scheme of $\quotstack{\Hom(S, K)}{\GL(S)}$. By Lemma~\ref{lem.coarsemodulispaces} this is the total space of the vector bundle $\Hom(\wedge^2 S, \wedge^2 K)$ over $\opensetinY\subset Y_2$, and now the section $(dW)_1$ is essentially the transpose of \eqref{eqn.transversederivativesofW}.  By this we mean: the transpose of $\eqref{eqn.transversederivativesofW}$ is a map of vector bundles $\Hom(\wedge^2 S, \wedge^2 K) \to \cN_{Y_2/\P^6}^\vee$, and the corresponding section of the pullback of $\cN_{Y_2/\P^6}^\vee$ to the total space of $\Hom(\wedge^2 S, \wedge^2 K)|_{Y'}$ is $(dW)_1$.  Since $\eqref{eqn.transversederivativesofW}$ is an isomorphism, this is a transverse section of the pullback of $\cN_{Y'/\P^6}^\vee$ that vanishes along the zero section of $\Hom(\wedge^2 S, \wedge^2 K)|_{Y'}$, so its Koszul complex is quasi-isomorphic to the structure sheaf of the zero section, and we conclude that
\[ \pi_*\RDerived\hom_{X_2|_U}(\cO_\maxiso,  \cO_\maxiso) \cong \cO_{\opensetinY}. \qedhere \]
\end{proof}

\begin{rem}\label{rem.localisotropicbundle} A bundle $L \subset \cO_{Y_2} \otimes V$ of maximally isotropic subspaces for the 2-forms $\twoform{p}$ may be constructed Zariski-locally on $Y_2$ as follows. Fix a point $[x] \in Y_1$; this determines a 2-dimensional subspace $\Im(x) \subset V$ which is isotropic for all $\twoform{p}$.  Then over the Zariski open set where $K_p \cap \Im(x) = 0$ we can take $L_p = K_p + \Im(x) \subset V$.  The complement of this open set, i.e.\ the locus where $K_p \cap \Im(x) \ne 0$, is a curve in $Y_2$.  We remark that this correspondence between points in $Y_1$ and curves in $Y_2$ is the essential ingredient of \cite{BorCal}.
\end{rem}

We do not know how to find such a bundle $L$ over the whole of $Y_2$, however, and indeed we suspect that no such global bundle exists. Consequently, we cannot immediately use the construction of Proposition~\ref{prop:maxisocalculation} to give a global generating object.  Fortunately we know another equivalent construction, one which does work globally, as we explain in the next section.

\subsection{Another construction of point-like objects}\label{sect:anotherconstruction}

Instead of using a maximally isotropic subbundle, we will use the following subspace:

\begin{defn}\label{defn:Gamma}
Let $\Gamma \subset X_2$ be the closed substack consisting of points $(x,p)$ where $[p]\in Y_2$, and the map
$$\bar x\colon S \to V/K_p $$
has rank at most 1.
\end{defn}
$\Gamma$ is a flat family of stacks over $Y_2$, as its fibres are all isomorphic. Observe that $W$ vanishes along $\Gamma$, and that $\Gamma$ is a cone in each fibre of $\pi$, hence is $\C^*_R$-invariant; therefore $\cO_\Gamma$ is a curved dg-sheaf on $X_2$, restricting on each fibre to give a curved dg-sheaf $\cO_{\Gamma_p}$ on $\stackFibre$. 

As we shall show momentarily, the object $\cO_{\Gamma_p}$ is (approximately) equivalent to $\cO_{M_p}$, where $M_p$ is a maximally isotropic subspace of $\stackFibre$ as in the previous section. The proof is a little involved, but let us first remark why the result is not so surprising.

The quadratic form $W_p$ on $\Hom(S,V)$ descends to a non-degenerate one $W_p'$ on $\Hom(S,V/K_p)$, so we have a pullback functor
$$\D(\Hom(S, V/K_p), W_p') \to \D(\Hom(S,V), W_p). $$
By definition, $\Gamma_p$ is the preimage of the locus of rank-1 matrices in $\Hom(S,V/K_p)$, and $M_p$ the preimage of the maximally isotropic subspace
$$ \Hom(S, L_p/K_p) \subset \Hom(S, V/K_p), $$
where $L_p/K_p$ is a Lagrangian in $V/K_p$. Consequently, both $\cO_{\Gamma_p}$ and $\cO_{M_p}$ are pullbacks of objects in $\D(\Hom(S, V/K_p), W_p')$.\footnote{We assume for this rough argument that both of these curved dg-sheaves are equivalent to matrix factorizations.}  But $W_p'$ is non-degenerate, so by Kn\"orrer periodicity this category is equivalent to the derived category of a point. It is hardly surprising, then, that two natural objects in this category turn out to be isomorphic.

\begin{prop}\label{prop.rank1vsmaximallyisotropic} Fix $[p]\in Y_2$.  Let $L_p\subset V$ be a maximally isotropic subspace for the 2-form $\twoform{p} = p \circ A$.  Then the curved dg-sheaf $\cO_{\Gamma_p}$  is homotopy-equivalent to the curved dg-sheaf
$$\cO_{\maxiso_p}\otimes \det S \otimes \det(L_p/K_p)^{-1}[-1]$$
in $\QCoh_{dg}(\stackFibre, W_p)$.
\end{prop}
The term $\det(L_p/K_p)^{-1}$ is a trivial line bundle on $\stackFibre$, but will be necessary later when we let $p$ vary. 

Recall from Section \ref{sect:catsofcurveddgsheaves} that $\QCohdg(\stackFibre, W)$ is the dg-category of curved dg-sheaves localized at `acyclic' objects, and that
$\Perf(X,W)$ is the full subcategory of objects that are equivalent to matrix factorizations. The curved dg-sheaf $\cO_{\maxiso_p}$ certainly lies in $\Perf(\stackFibre, W_p)$, so this proposition proves incidentally that $\cO_{\Gamma_p}\in \Perf(\stackFibre, W_p)$ too.
\begin{proof}
Consider the locus
\[ \Sigma_p := \setconds{x \in \Hom(S,V)}{W_p(x) = 0,\, \dim(L_p + \Im(x)) \le 6}. \]
It contains both $\maxiso_p$ and $\Gamma_p$.  It's an intersection of two quadrics in $\Hom(S,V)$: one cut out by $W_p$ and the other by the determinant of the $2\times2$ matrix 
$$S \xrightarrow{x} V \to V/L_p.$$
In fact it is a complete intersection: $W_p$ is a quadric of rank 10, hence is irreducible, and the second quadric has rank at most 4, hence is different from $W_p$, so their intersection is complete.

Thus $\cO_{\Sigma_p}$ is the restriction to $\set{W_p=0}$ of a perfect complex on $\stackFibre$, and hence it is a contractible curved dg-sheaf by \S \ref{rem:basicproperties}(\ref{itm:dsing}). So to prove the lemma it is enough to show the equivalence of the ideal sheaves
\begin{equation}\label{eq:equivofidealsheaves}I_{\Gamma_p/\Sigma_p}\cong I_{\maxiso_p/\Sigma_p} \otimes \det S \otimes \det(L_p/K_p)^{-1}[1]\end{equation}
as curved dg-sheaves on $(\stackFibre, W_p)$. 

We take the following ($\SL(S)\times\C^*_R$)-equivariant resolution of singularities of $\Sigma_p$:
\[ \tilde\Sigma_p := \setconds{(x,l,H) \in \Hom(S,V) \times \P S \times \Gr(6,V)}{(L_p + \Im x) \subset H,\,x(l) \subset H^\perp}, \]
with the evident projection map
$$\pr_1\colon \tilde\Sigma_p \to \Sigma_p.$$
Here the orthogonal $H^\perp$ is taken with respect to the pairing $\twoform{p}$, and since $L_p \subset H$ we have $H^\perp \subset L_p^\perp = L_p \subset H$.  To see that $\tilde\Sigma_p$ is smooth, observe that the projection
\[ \pr_{23}\colon \tilde\Sigma_p \to \P S \times \P(V/L_p) \cong \P^1 \times \P^1\]
defined by $\pr_{23}(x,l,H) = (l, H/L_p)$ is a vector bundle: the fibre is the 10-dimensional vector space
\[ \setconds{x \in \Hom(S,H)}{x(l) \subset H^\perp}. \]
To see that $\pr_1$ is a resolution of singularities, observe that if $x\in \Sigma_p$ is generic in the sense that $\Im(x)\not\subset L_p$, then the fibre $\pr_1^{-1}(x)$ is a single point: clearly $H$ is uniquely determined, but also $x^{-1}(H^\perp)$ must be 1-dimensional, and this determines $l$.

Now we analyze the non-generic fibres of $\pr_1$, over points $x$ where $\Im(x)\subset L_p$. There are three cases:
\begin{itemize}
\item $\dim(\Im(x) + K_p)=5$. Then $x$ has rank 2. We can choose $H$ freely, and then we must set $l = x^{-1}(H^\perp)$. Thus the fibre is $\P(V/L_p) \cong \P^1$. 

\item $\dim(\Im(x) + K_p)=4$.  The fibre has two irreducible components: either $H = (\Im(x) + K_p)^\perp$ and we can choose $l$ freely, or $l = x^{-1}(K_p)$ and we can choose $H$ freely. Thus the fibre is two copies of $\P^1$ meeting at a point: precisely, $\P S$ meeting $\P(V/L_p)$ at the point $(x, x^{-1}(K_p), (\Im(x) + K_p)^\perp)$.

\item $\dim(\Im(x) + K_p)=3$. We can choose both $l$ and $H$ freely, so the fibre is $\P S \times \P(V/L_p) \cong \P^1 \times \P^1$.
\end{itemize}
Consequently $\RDerived\pr_{1*} \cO_{\tilde\Sigma_p} = \cO_{\Sigma_p}$, i.e.\ $\Sigma_p$ has rational singularities.

Next we consider the preimage of $\maxiso_p$ in $\tilde\Sigma_p$:
$$\tilde \maxiso_p := \setconds{(x,l,H) \in \tilde\Sigma_p}{\Im(x) \subset L_p}.$$
This is the union of all the non-generic fibres.  The projection $\pr_{23}$ makes $\tilde \maxiso_p$ into a rank-9 vector bundle over $\P^1 \times \P^1$, so $\tilde \maxiso_p$ is smooth.  From the above analysis of the fibres we know that $\RDerived\pr_{1*} \cO_{\tilde \maxiso_p} = \cO_{\maxiso_p}$.  Also $\tilde \maxiso_p \subset \tilde\Sigma_p$ is a divisor, and it's the zero locus of the map
$$S/l = \det S\otimes l^{-1} \; \stackrel{x}{\longrightarrow}\; H/L_p $$
which is a section of the line bundle $\pr_{23}^* \cO(-1,-1) \otimes \det S^{-1}$. So we have an exact sequence
\[ 0 \to \pr_{23}^* \cO(1,1) \otimes \det S[-1] \to \cO_{\tilde\Sigma_p} \to \cO_{\tilde \maxiso_p} \to 0. \]
The R-charge shift occurs because the map $x$ has R-charge 1. Applying $\RDerived\pr_{1*}$ to the above exact sequence gives us
\[ \RDerived\pr_{1*} \pr_{23}^* \cO(1,1) \otimes \det S = I_{\maxiso_p/\Sigma_p}[1]. \]

The final variety we consider is the proper transform of $\Gamma_p$ in $\tilde\Sigma_p$:
\[ \tilde\Gamma_p := \set{ (x,l,H) \in \tilde\Sigma_p :\; x(l) \subset K_p }. \]
The projection $\pr_{23}$ makes $\tilde\Gamma_p$ into a rank-9 vector bundle over $\P^1 \times \P^1$, so it too is smooth, and a similar inspection of the fibres of $\pr_1$ yields $\RDerived \pr_{1*} \cO_{\tilde\Gamma_p} = \cO_{\Gamma_p}$. The subvariety $\tilde\Gamma_p \subset \tilde\Sigma_p$ is also a divisor.  It is the zero locus of the map
$$l  \; \stackrel{x}{\longrightarrow}\;( H^\perp/K_p)\cong (H/L_p)\otimes \det S^{-1}\otimes \det (L_p/K_p) $$
(to see the equality here note that the 2-form $\twoform{p}$ takes values in $\det S$, so $H^\perp$ is the kernel of the map $V \to H^* \otimes \det S$ induced by $\twoform{p}$, from which it follows that $\det(H^\perp) = \det H \otimes \det S^{-1}$).  Thus $\tilde\Gamma_p$ is cut out by a section of the line bundle
$$\pr_{23}^* \cO(1,-1) \otimes \det S^{-1} \otimes \det (L_p/K_p)$$
having R-charge 1. We take the exact sequence
\[ 0 \to \pr_{23}^* \cO(-1,1)\otimes \det S\otimes \det (L_p/K_p)^{-1}[-1] \longrightarrow \cO_{\tilde\Sigma_p} \longrightarrow \cO_{\tilde\Gamma_p} \to 0 \]
and apply $\RDerived\pr_{1*}$ to get
\[ \RDerived\pr_{1*} \pr_{23}^* \cO(-1,1) = I_{\Gamma_p/\Sigma_p}\otimes \det S^{-1}\otimes \det (L_p/K_p)[1]. \]

Next, take the exact sequence of bundles on $\P S \times \P(V/L_p)$
\[ 0 \to \cO(-1,1) \to \cO(0,1) \otimes S \to \cO(1,1) \otimes \det S \to 0 \]
and apply $\RDerived\pr_{1*} \pr_{23}^*$ to get an exact triangle on $\Hom(S,V)$:
\[ I_{\Gamma_p/\Sigma_p}\otimes \det S^{-1}\otimes \det(L_p/K_p)[1] \longrightarrow \RDerived\pr_{1*} \pr_3^* \cO(1) \otimes S \longrightarrow 
I_{\maxiso_p/\Sigma_p}[-1], \]
where $\pr_3: \tilde\Sigma_p \to \P(V/L_p)$ is projection onto the third component.  Thus the claim \eqref{eq:equivofidealsheaves} reduces to the claim that $\RDerived\pr_{1*} \pr_3^* \cO(1)$ is a contractible curved dg-sheaf.

To prove this last claim, take the exact sequence of sheaves on $\P(V/L_p)$
\[ 0 \to \cO \to \cO(1) \to \cO_{H_0/L_p} \to 0, \]
where $\cO_{H_0/L_p}$ is the skyscraper sheaf at some point $H_0/L_p \in \P(V/L_p)$, and apply $\RDerived\pr_{1*} \pr_3^*$.  Observe that $\pr_3$ is flat since $\pr_{23}$ is.   We get an exact triangle on $\Hom(S,V)$:
\[ \cO_{\Sigma_p} \to \RDerived\pr_{1*} \pr_3^* \cO(1) \to \RDerived \pr_{1*} \pr_3^* \cO_{H_0/L_p}. \]
We know that the first term is trivial in $\D(\stackFibre, W_p)$, so to show that the second is trivial it is enough to show that the third is.  Analyzing fibres again we find that $\RDerived\pr_{1*} \pr_3^* \cO_{H_0/L_p}$ is the structure sheaf of the locus 
$$\setconds{x \in \Hom(S,V)}{W_p(x) = 0,\, \Im(x) \subset H_0}.$$
 This is the complete intersection of the quadric cut out by $W_p$ with the two hyperplanes, so its structure sheaf is indeed a contractible curved dg-sheaf (by
\S \ref{rem:basicproperties}(\ref{itm:dsing}) again).
\end{proof}

Now suppose we have a family of maximally isotropic subspaces, as in Proposition~\ref{prop:maxisocalculation}.

\begin{cor}
Let $U\subset \P^6$ be a Zariski open set, let $Y' = U\cap Y_2$, and let 
$$j\colon X_2|_Y'\into X_2|_U$$
 denote the inclusion. Assume that  over $Y'$ we have a bundle $L\subset \cO_{Y'}\otimes V$ of maximally isotropic subspaces for $\omega$, and let $M=\Hom(S,L)\subset X_2|_{Y'}$ be the corresponding bundle of maximally isotropic subspaces for $W$.  Then (possibly after shrinking $U$) we have a homotopy-equivalence of curved dg-sheaves
$$\cO_{\Gamma|_{Y'}} \; \simeq \; j_*\!\left(\cO_{\maxiso}\otimes \det S \otimes \det(L/K)^{-1}\right)[-1] $$
in  $\QCoh_{dg}(X_2|_U, W)$.
\end{cor}
\begin{proof}
Assume $U$ is small enough that $X_2|_U \cong \cF\times U$, so $X_2|_{Y'} = \cF\times Y'$. Then the proof of Proposition~\ref{prop.rank1vsmaximallyisotropic} works perfectly well over the base $Y'$. The only point to note is that we also need to pick a bundle $H_0$ of co-isotropics containing $L$: we can certainly do this if $U$ is small enough. Consequently these objects are equivalent as curved dg-sheaves on $(X_2|_{Y'}, W)$, and applying the functor $j_*$ we deduce the statement of the corollary.
\end{proof}

So $\cO_\Gamma$ gives us a global version of our generating object, but in local patches we can continue to work with maximally isotropic subspaces.

\begin{rem}
\emph{A priori} it might seem simpler to just use $\cO_\Gamma$, and ignore the maximally isotropics entirely. Unfortunately, since $\Gamma$ is singular, it is prohibitively difficult to do the calculation of Proposition~\ref{prop:maxisocalculation} directly for $\cO_\Gamma$.
\end{rem}

\subsection{Completing the proof} \label{sec:finale}

For every $[p]\in Y_2$, we have a curved dg-sheaf $\cO_{\Gamma_p}$ supported on the fibre over $p$, and these fit into a global family $\cO_\Gamma$. We want to consider a functor
 $$F\colon \D(Y_2) \to \D(X_2, W)$$
which sends each $\cO_{[p]}$ to the corresponding $\cO_{\Gamma_p}$, i.e.\ the functor which has $\cO_\Gamma$ as its Fourier--Mukai kernel.  From the results in the previous two sections we know that each $\cO_{\Gamma_p}$ behaves like the corresponding skyscraper sheaf $\cO_{[p]}\in \D(Y_2)$, and this more-or-less guarantees that $F$ will be an embedding (cf.~Remark~\ref{rem:embeddings}). In this section we fill in the remaining details in this argument and then show that the image of this embedding is exactly the category $\Br(X_2, W)$ from Section~\ref{sect:windows}.

First we give the definition of $F$ in full. We consider the diagram
$$\begin{tikzcd}  \Gamma \arrow{r}{i} \arrow{dr}[swap]{\hat{\pi}}  &     X_2|_{Y_2} \arrow{r}{j} \arrow{d}{\pi}    &    X_2 \arrow{d}{\pi} \\
               &       Y_2      \arrow{r}[swap]{j}           &      \P^6
\end{tikzcd} $$
and define
$$ F := j_* i_*\hat{\pi}^* \colon \D(Y_2) \to [\QCohdg(X_2, W)] $$
where $[\QCohdg(X_2,W)]$ denotes the homotopy category of $\QCohdg(X_2, W)$. Since $\Gamma$ is flat over $Y_2$, we see that $F$ sends $\cO_{[p]}$ to $\cO_{\Gamma_p}$, and it sends $\cO_{Y_2}$ to $\cO_\Gamma$. 

 We claim that $F$ lands in the subcategory $\D(X_2, W) \subset  [\QCohdg(X_2, W)]$, \emph{i.e.} everything in the image is homotopy-equivalent to a matrix factorization.  More importantly, $F$  in fact lands in the `window' subcategory
$$\Br(X_2, W) \subset \D(X_2, W).  $$
Recall from Section~\ref{sect:windows} that this is the subcategory
where we only allow (objects homotopy-equivalent to) matrix factorizations built out of a certain set of vector bundles, namely the ones corresponding to the `rectangle' \eqref{eq:rectangularcollection} in the irreducible representations of $\GL(S)$.

\begin{prop}\label{prop.Flandsinwindow} For all $\cE\in \D(Y_2)$, we have $F\cE \in \Br(X_2, W)$.
\end{prop}
\begin{proof}
It's enough to prove the statement when $\cE$ is a sheaf. In that case $F\cE$ is a sheaf on $X_2$, which we can write as
$$F\cE  = j_*i_*\hat{\pi}^*\cE = j_*(\cO_\Gamma\otimes \pi^*\cE).$$

The sheaf $\cO_\Gamma$ on $X_2|_{Y_2}$ has an Eagon--Northcott resolution (e.g.\ \cite[\S 6.1.6]{Weyman})
$$ 0 \to \wedge^4(V/K)^\vee\otimes \Sym^2 S(1) \to  \wedge^3(V/K)^\vee\otimes S(1)
\to  \wedge^2(V/K)^\vee(1) \to \cO \to \cO_\Gamma \to 0, $$
and we can make this $\C^*_R$-equivariant by inserting the necessary shifts.
Consequently, the sheaf $F\cE$ has a $\C^*_R$-equivariant resolution on $X_2$ of the form 
$$0 \to \pi^*\cF_3\otimes \Sym^2 S \to \pi^*\cF_2\otimes S \to \pi^*\cF_1 \to \pi^*\cF_0 \to F\cE \to 0 $$
where $\cF_0, \ldots, \cF_3$ are sheaves on $\P^6$, supported on $Y_2$. Every sheaf on $\P^6$ can be resolved by the line bundles $\cO, \ldots, \cO(6)$, so $F\cE$ has a resolution by vector bundles lying in our rectangle \eqref{eq:rectangularcollection}.
 Now Lemma~\ref{lem:resolvingintowindows} implies that $F\cE$ lies in the subcategory $\Br(X_2, W)\subset  \D(X_2, W)$.
\end{proof}

Next we need to establish that $F$ has a right adjoint. The functor
$$ (ji)_*\colon \D(\Gamma) \to \D(X_2, W) $$
has a right adjoint, namely
$$(ji)^!=\RDerived\hom_{X_2}(\cO_\Gamma, -) \colon \D(X_2, W) \to \D(\Gamma).$$
Note that this statement is local, and that locally $\cO_\Gamma$ generates $\D(\Gamma)$, so to prove the adjunction it's sufficient to observe that
$$\RDerived\hom_{X_2}((ji)_* \cO_\Gamma, \cF) = \RDerived\hom_\Gamma(\cO_\Gamma, \RDerived\hom_{X_2}(\cO_\Gamma, \cF) ) $$
holds tautologically for any $\cF\in \D(X_2, W)$. 

The right adjoint to the functor
$$\hat{\pi}^*\colon \D(Y_2) \to \D(\Gamma) $$
should be $\hat{\pi}_*$, but unfortunately $\Gamma$ is not proper (not even equivariantly), so $\hat{\pi}_*$ produces quasi-coherent sheaves in general. Fortunately, we have the following.

\begin{lem}\label{lem.adjointexists} 
 For $\cE\in \D(X_2, W)$, the complex of sheaves $\pi_*\RDerived\hom_{X_2}(\cO_\Gamma, \cE)$ has bounded and coherent homology sheaves. Consequently
$$ F^R:= \pi_*\RDerived\hom_{X_2}(\cO_\Gamma, -) \colon \D(X_2, W) \to \D(Y_2) $$
is right adjoint to $F$.
\end{lem}
\begin{proof}
The homology of $\RDerived\hom_{X_2}(\cO_\Gamma, \cE)$ is a coherent sheaf whose support lies in the critical locus $\Crit(W)$ of $W$ (see \S \ref{rem:basicproperties}(\ref{itm:support})), and also in $\pi^{-1}(Y_2)$. From Lemma~\ref{lem.coarsemodulispaces}, the map
$$\pi\colon \Crit(W)|_{Y_2} \to Y_2 $$
is just passage to the underlying scheme, so if $\cF$ is a sheaf supported in $\Crit(W)$ then (locally in $Y_2$) we calculate $\pi_*\cF$ by just taking $SL(S)$-invariants.  If $Y'\subset Y_2$ is a sufficiently small open set, then over $\pi^{-1}(Y')$ we can find a finite presentation of the homology of $\RDerived\hom_{X_2}(\cO_\Gamma, \cE)$. Taking $\SL(S)$-invariants, which is an exact functor, gives us a finite presentation of the homology of $\pi_*\RDerived\hom_{X_2}(\cO_\Gamma, \cE)$.
\end{proof}

From their definitions, both $F$ and $F^R$ are `local' over $\P^6$, i.e.\ linear over the structure sheaf of $\P^6$. Consequently if $\opensetinP \subset \P^6$ is an open set then we have functors
$$\begin{tikzcd} \D(Y_2\cap \opensetinP) \arrow[transform canvas={yshift=0.4ex}]{r}{F} & \D(X_2|_{\opensetinP}, W). \arrow[transform canvas={yshift=-0.4ex}]{l}{F^R}\end{tikzcd}$$

\begin{prop}\label{prop.Fanembedding}
The functor $F$ is fully faithful.
\end{prop}
\begin{proof}
We will show that for any $\cE\in \D(Y_2)$, the unit of the adjunction
$$\cE\to F^RF\cE$$
 is an isomorphism. Then the composition $F^RF$ is naturally isomorphic to the identity functor, and so $F$ must be an embedding.

 This statement is local in $Y_2$, so we can restrict to an affine open subset $\opensetinP\subset \P^6$ and corresponding open set $\opensetinY = Y_2\cap\opensetinP$. Then it's enough to check the statement on the structure sheaf $\cO_{\opensetinY}$, since this generates $\D(\opensetinY)$. So the required statement is that
$$F\colon \cO_{\opensetinY} \to \pi_*\RDerived\hom_{X_2|_U}(\cO_{\Gamma_{\opensetinY}}, \cO_{\Gamma_{\opensetinY}}) $$
is a quasi-isomorphism. 

By making $\opensetinP$ smaller if necessary, we may assume that we have a bundle of maximally isotropics $L \subset V_{\opensetinY}$, and an associated  bundle $\maxiso = \Hom(S,L) \subset X_2|_{\opensetinY}$. Then  by Proposition~\ref{prop.rank1vsmaximallyisotropic} we may replace $\cO_{\Gamma_{\opensetinY}}$ with $\cO_\maxiso$, up to a shift and twisting by a line bundle. Then  
$$\pi_*\RDerived\hom_{X_2|_U}(\cO_{\Gamma_{\opensetinY}}, \cO_{\Gamma_{\opensetinY}}) \cong \cO_{\opensetinY}$$
by Proposition~\ref{prop:maxisocalculation}. Finally, $F$ must be an isomorphism on homology because it must send the constant section 1 to itself (it preserves identity arrows), and it is linear over sections of $\cO_{\P^6}$.
\end{proof}

\begin{thm}\label{thm:Fisanequivalence} The functor
$$F\colon \D(Y_2) \to \Br(X_2, W) $$
is an equivalence.
\end{thm}
\begin{proof}
By Proposition~\ref{prop.Fanembedding} and Proposition~\ref{prop.Flandsinwindow} we have that $F$ is an embedding from $\D(Y_2)$ into $\Br(X_2, W)$, and by Lemma~\ref{lem.adjointexists} it has a right adjoint. We also know that $\Br(X_2, W)$ is equivalent to $\D(Y_1)$, by composing our equivalences $\Psi_1$ and $\Psi_2$. However $Y_1$ is Calabi--Yau and connected, so $\Br(X_2, W)$ cannot have a non-trivial admissible subcategory, and we deduce the result.
\end{proof}

So the equivalence $\Psi_3$ holds. This last step of the argument is rather unsatisfactory in that we have to appeal to our other two equivalences, rather than giving a self-contained proof. But presumably it is possible to prove directly that $\Br(X_2, W)$  is Calabi--Yau and connected -- in particular the Calabi--Yau property should follow by an argument along the lines of \cite[\S 4]{LP}.

\begin{rem}\label{rem:pfaffianforotherdimensions}

We conclude with some remarks about how our results adapt when we change the dimensions of $S$ and $V$ to $r$ and $d$ respectively.
\begin{itemize}
\item $r=2$, $d=5$. In this case $Y_1$ is an elliptic curve and $Y_2$ is the dual elliptic curve. We have a very similar definition of $\Br(X_2, W)$ (see Remark~\ref{rem:windowsinotherdims}), we have equivalences $\D(Y_1) \cong \D(X_1,W) \cong \Br(X_2,W)$ as before, and the methods of this section can be used to show that $\D(Y_2) \cong \Br(X_2, W)$. In fact this case is rather easier than the $d=7$ case because it's very easy to show that we have a global maximally isotropic subbundle $L$ on $Y_2$, and so we don't need any alternative construction as in Section~\ref{sect:anotherconstruction}.

\item $r=2$, $d=6$. In this case $Y_1$ is a K3 surface and $Y_2$ is a Pfaffian cubic 4-fold. We can define $\Br(X_2, W)$ as in Remark~\ref{rem:windowsinotherdims}, and the arguments of Sections \ref{sect:knorrer} and \ref{sect:windows} show that $\D(Y_1)$ embeds into $\Br(X_2, W)$. The methods of this section apply essentially verbatim to show that $\D(Y_2)$ also embeds into $\Br(X_2, W)$; the only change is that the bundle $K$ over $Y_2$ now has rank 2, and so
$$ \Hom(S, K) / \GL(S)  \cong \Hom(\wedge^2 S, \wedge^2 K)$$
is still a vector bundle, but of rank 1. Fortunately $Y_2$ is now codimension 1 in $\P^5$, and so \eqref{eqn.transversederivativesofW} is still an isomorphism. Unfortunately we do not have a proof that this second embedding is actually an equivalence, \emph{i.e.}~ we have no analogue of Theorem \ref{thm:Fisanequivalence}. If we did, we would recover Kuznetsov's result that $D^b(Y_1)$ embeds into $\D(Y_2)$ \cite[Thm.~2]{KuznetsovHPDlines}.

\item $r=2$, $d=9$. In this case both $Y_1$ and $Y_2$ are smooth Calabi-Yau 5-folds, and we have two-thirds of a proof that they are derived equivalent. As before we have $\D(Y_1)\cong \D(X_1, W)\cong \Br(X_2, W)$, and Proposition \ref{prop:maxisocalculation} still holds, but we do not know how to construct a kernel globally over $Y_2$ (since $\cO_\Gamma$ only works for $d=6$ or $7$).

\item $r=2$, $d = 8$ or $d>9$.  We do have a category $\Br(X_2,W)$, but $Y_2$ is necessarily singular, so our calculations with maximally isotropic subspaces show that $\Br(X_2, W)$ is in some sense a non-commutative resolution of $\D(Y_2)$.  Indeed, we speculate that the homological projective dual to $\Gr(2,V)$ is the non-commutative resolution of the Pfaffian locus $\Pf \subset \P\Hom(\Wedge^2 V, \Wedge^2 S)$ constructed as follows: take the stack
\[ \quotstackBig{ \setcondsbig{(x,\omega) \in \Hom(S,V) \oplus \Hom(\wedge^2 V,\wedge^2 S)}{\omega \ne 0}}{\GL(S)} \]
with the superpotential
\[ W(x,\omega) = \omega \circ \wedge^2 x, \]
and take the subcategory of matrix factorizations built from the vector bundles
\[ \WindowSet{ \Sym^l S^\vee \otimes (\det S^\vee)^m}{\tfrac{1}{2}(d-1)}{\binom{d}{2}} \]
when $d$ is odd, or
\[ \WindowSet{ \Sym^l S^\vee \otimes (\det S^\vee)^m}{\tfrac{1}{2}d}{\binom{d}{2}} \]
when $d$ is even.  This line of inquiry is currently being pursued by Ballard et al.\ \cite{BDFIK}.  Of course one would like to begin by checking that this is equivalent to Kuznetsov's non-commutative resolution of $\Pf$ when $d = 6$ and 7 \cite{KuznetsovHPDlines}.

\item If $r > 2$ then it is not clear to us how to proceed.

\end{itemize}
\end{rem}

\end{document}